\documentclass[11pt]{amsart}
\usepackage{amsmath}
\usepackage{amssymb}
\usepackage{amsthm,amscd}
\usepackage{latexsym}
\usepackage{hyperref}
\usepackage{enumerate}
\usepackage[all]{xy}
\usepackage{tikz}

\usepackage[all]{xy} \SelectTips {cm}{}
\usepackage{tikz}
\usetikzlibrary{matrix,arrows,decorations.pathreplacing,decorations.pathmorphing}
\usepackage{amsmath,amsthm,amssymb,amscd,url}

\usepackage{verbatim}
\usepackage[all]{xy} \SelectTips {cm}{}
\usepackage{tikz}
\usetikzlibrary{matrix,arrows,decorations.pathreplacing,decorations.pathmorphing}

\newtheorem{thm}{Theorem}[section]
\newtheorem{cor}[thm]{Corollary}
\newtheorem{lem}[thm]{Lemma}
\newtheorem{prop}[thm]{Proposition}

\newtheorem{cond}[thm]{Condition}

\newtheorem{thmintro}{Theorem}
\newtheorem{ex}[thm]{Example}

\newcommand{\enuma}[1]{\begin{enumerate}[\textup{(}a\textup{)}] {#1} \end{enumerate}}

\newcommand{\mb}{\mathbf}
\newcommand{\bW}{{\mathbf W}}

\newcommand{\mr}{\mathrm}
\newcommand{\mc}{\mathcal}
\newcommand{\mf}{\mathfrak}

\newcommand{\N}{\mathbb N}
\newcommand{\Z}{\mathbb Z}

\newcommand{\R}{\mathbb R}
\newcommand{\C}{\mathbb C}
\newcommand{\LLC}{\mathrm{LLC}}

\newcommand{\matje}[4]{\left(\begin{smallmatrix} #1 & #2 \\ 
#3 & #4 \end{smallmatrix}\right)}

\newcommand{\q}{/\!/}

\def\Hom{{\rm Hom}}
\def\End{{\rm End}}

\def\Irr{{\mathbf {Irr}}}

\def\GL{{\rm GL}}

\def\PGL{{\rm PGL}}

\def\SL{{\rm SL}}

\def\cH{{\mathcal H}}

\def\cO{{\mathcal O}}

\def\Ind{{\rm Ind}}
\def\ind{{\rm ind}}

\def\Mod{{\rm Mod}}
\def\nr{{\rm nr}}

\def\fs{{\mathfrak s}}
\def\ft{{\mathfrak t}}
\def\fB{{\mathfrak B}}

\def\Rep{{\rm Rep}}
\def\Mod{{\rm Mod}}
\def\Res{{\rm Res}}
\def\Stab{{\rm Stab}}
\def\Nrd{{\rm Nrd}}

\def\Nor{{\rm N}}

\def\der{{\rm der}}

\def\temp{{\rm temp}}

\def\op{{\rm op}}

\def\cusp{{\rm cusp}}
\def\uni{{\rm un}}

\begin{document}

\title[Inner forms]{Smooth duals of inner forms of $\GL_n$ and $\SL_n$}

\author[A.-M. Aubert]{Anne-Marie Aubert}
\address{CNRS, Sorbonne Universit\'e, Universit\'e Paris Diderot, Institut de Math\'ematiques de Jussieu -- Paris Rive Gauche, IMJ-PRG, F-75005, Paris, France}
\email{anne-marie.aubert@imj-prg.fr}
\author[P. Baum]{Paul Baum}
\address{Mathematics Department, Pennsylvania State University, University Park, PA 16802, USA}
\email{pxb6@psu.edu}
\author[R. Plymen]{Roger Plymen}
\address{School of Mathematics, Manchester University, Manchester M13 9PL, England}
\email{roger.j.plymen@manchester.ac.uk}
\author[M. Solleveld]{Maarten Solleveld}
\address{IMAPP, Radboud Universiteit Nijmegen, Heyendaalseweg 135, 
6525AJ Nijmegen, the Netherlands}
\email{m.solleveld@science.ru.nl}
\date{\today}
\subjclass[2010]{20G25, 22E50}
\thanks{
The fourth author is supported by a NWO Vidi grant "A Hecke algebra approach to the 
local Langlands correspondence" (nr. 639.032.528).\\
We thank the referee for the exceptionally detailed and accurate report.
}
\maketitle

\begin{abstract} Let $F$ be a local non-archimedean field.   
We prove that every Bernstein component in the smooth dual of each inner form of the general linear group $\GL_n(F)$ 
is canonically in bijection with the extended quotient for the action, given by Bernstein, of a finite group on a complex torus.  

For inner forms of $\SL_n(F)$ we prove that each Bernstein component is canonically in bijection with the associated twisted extended quotient.

In both cases, the bijections satisfy naturality properties with respect to the tempered dual, parabolic induction, central character, 
and the local Langlands correspondence.

\end{abstract}

\tableofcontents

\section*{Introduction}   
Let $X$ be a complex affine variety.  Denote the coordinate algebra of $X$ by $\cO(X)$.   
The Hilbert Nullstellensatz asserts that $X \mapsto \cO(X)$ is an equivalence of categories
\vspace{3mm}\\
\begin{center}
\begin{tikzpicture}[node distance=.5cm]
\node[text width=4.3cm, left delimiter=(, right delimiter=),left] at (-.6,0)
{unital commutative finitely generated nilpotent-free \mbox{$\C$-algebras}};
\node at (0,0){$\sim$};
\node[text width=3.1cm,left delimiter=(, right delimiter=),right] at (.6,0)
{affine complex \\ algebraic varieties};
\node at (4.5,.5) {op};
\end{tikzpicture}\\
$\mathcal{O}(X)\mapsto X$
\end{center}
\vspace{4mm}
where $\op$ denotes the opposite category.   

A finite type algebra is a $\C$-algebra    $A$ with a given structure as an $\cO(X)$-module such that $A$ is finitely 
generated as an $\cO(X)$-module.  A compatibility is required between the algebra structure of $A$ and 
the given action of $\cO(X)$ on $A$. However, $A$ is not required to be unital.   Due to the above equivalence of categories, any finite type algebra 
can be viewed as a slightly non-commutative affine algebraic variety.  This will be the point of view of the paper.

Following this point of view, each Bernstein component in the smooth dual of any reductive $p$-adic group $G$ is a non-commutative
affine algebraic variety.   In a series of papers we have examined the question ``What is the geometric structure of any given Bernstein component?"
Our proposed answer to this (which has been verified for all the classical split reductive $p$-adic groups) is based on the notion of \emph{extended quotient}.   In the present
paper we show that for inner forms of $\SL_n$ it is necessary to use a \emph{twisted extended quotient} (see the Appendix).  The twisting is given by a family of 
$2$-cocycles.  

Let $F$ be a local non-archimedean field.   Let $D$ be a central simple $F$-algebra with $\dim_F (D) = d^2$. 
Then $\GL_m (D)$ is an inner form of $\GL_{md}(F)$ and the derived group $\GL_m (D)_\der$ is an inner form of $\SL_{md}(F)$.  The main result of this paper is

\begin{thmintro}\label{conj:1} 
Let $G$ be an inner form of $\GL_n(F)$ or $\SL_n(F)$.   
Let $\Irr^\fs(G)$ be any Bernstein component in the smooth dual of $G$.  Let
$T_\fs \q W_\fs$ and $(T_\fs \q W_\fs)_{\natural}$ be
the appropriate extended quotient and twisted extended quotient. 
Then:
\begin{itemize} 
\item[{\rm (1)}] If $G$ is an inner form of $\GL_n(F)$  there exists a 
bijection
\begin{equation*}\label{eq:3}
\Irr^\fs (G) \longleftrightarrow T_\fs \q W_\fs. 
\end{equation*}
\item[{\rm (2)}] If $G$ is an inner form of $\SL_n(F)$ there exists a
family of 2-cocycles $\natural$ and a bijection
\begin{equation*}
\Irr^\fs (G) \longleftrightarrow (T_\fs \q W_\fs )_\natural.
\end{equation*}
\item[{\rm (3)}] In either case, the bijection satisfies  naturality properties with respect to the tempered dual, parabolic induction, and central character.
\end{itemize}
\end{thmintro}


\smallskip

\noindent
We remark that

\noindent
$\bullet$
 $T_\fs \q W_\fs$ is the non-commutative affine variety whose coordinate algebra is the crossed product algebra
$\cO(T_\fs) \rtimes W_\fs$ where $T_\fs, W_\fs$ are respectively the complex torus and finite group acting on the torus which Bernstein assigns to 
$\Irr^\fs(G)$.   The crossed product algebra  $\cO(T_\fs) \rtimes W_\fs$ is a finite $\cO(T_\fs/W_\fs)$-algebra.

\noindent
$\bullet$
  $(T_\fs \q W_\fs)_{\natural}$ is the non-commutative affine variety whose coordinate algebra is the twisted crossed product algebra 
$\cO(T_\fs) \rtimes_{\natural} W_\fs$.     The twisted crossed product algebra  $\cO(T_\fs) \rtimes_{\natural} W_\fs$ is a finite $\cO(T_\fs/W_\fs)$-algebra.  
Example~\ref{the_example} shows that for inner forms of $\SL_5$ there are Bernstein components where the twisting is non-trivial.

We observe that it is somewhat remarkable that many of the subtleties of the representation theory of $\GL_m(D)$, $\SL_m(D)$ are captured by the algebras  
  \[
  \cO(T_\fs) \rtimes W_\fs \quad \quad 
   \cO(T_\fs) \rtimes_{\natural} W_\fs
  \]
and  their geometric realizations  $T_\fs \q W_\fs, (T_\fs \q W_\fs)_{\natural}$.

\medskip

The proof of Theorem  \ref{conj:1} uses a weakening of Morita equivalence called \emph{stratified equivalence}, see \cite{ABPS8}.   The proof is achieved by combining the theory of types 
with an analysis of the structure of Hecke algebras.   

\noindent
\emph{Outline of the proof}.

The proof of Theorem \ref{conj:1} for $\GL_m(D)$ consists of two steps.   
\begin{itemize}
\item Step 1:  For each point $\fs$ in the Bernstein 
spectrum $\fB(\GL_m(D))$, the ideal $\cH(G)^\fs$ in the Hecke algebra $\cH(\GL_m(D))$ is Morita equivalent to an affine Hecke algebra, see \cite{ABPS4}.    
\item Step 2:  Using the Lusztig asymptotic algebra, a stratified equivalence is constructed between this affine Hecke algebra and  the associated crossed product 
algebra.  The irreducible representations of the crossed product algebra (in a canonical way) are the required extended quotient. 
\end{itemize} 

The proof of Theorem \ref{conj:1} for $ \SL_m(D)$ is achieved by introducing an intermediate group  
$\SL_m(D)\cdot Z(\GL_m(D))$:
\[
\SL_m(D) \subset \SL_m(D)\cdot Z(\GL_m(D)) \subset \GL_m(D),
\]
where $Z(\GL_m(D))$ is the center of $\GL_m(D)$. It consists in two analogous steps.
\begin{itemize}
\item Step 1: For each point $\fs$ in the Bernstein spectrum $\fB(\SL_m(D))$, the ideal $\cH(\SL_m(D))^\fs$ 
in the Hecke algebra $\cH(\SL_m(D))$ is Morita equivalent to a twisted affine Hecke algebra.  
\item Step 2: Using the Lusztig asymptotic algebra, a stratified equivalence is constructed between this twisted affine Hecke algebra and  the associated 
twisted crossed product algebra.  The irreducible representations of the twisted crossed product algebra (in a canonical way) are the required 
twisted extended quotient. 
\end{itemize}    

\medskip
In Section~\ref{sec:LLC}, we prove that the bijections of Theorem~\ref{conj:1} are compatible with 
the local Langlands correspondence, in the sense below.
Let $G$ be an inner form of $\GL_n(F)$ or $\SL_n(F)$, and let $\check G$ denote $\GL_n(\C)$ or 
$\PGL_n(\C)$, respectively. The set of $\check G$-conjugacy classes of Langlands parameters 
(resp. enhanced Langlands parameters) for $G$ is denoted by $\Phi(G)$ (resp. $\Phi_e(G)$). 
We may identify $\Phi_e(G)$ with $\Phi(G)$ when $G$ is an inner form of $\GL_n(F)$. 

Let $\mathcal L$ be a set of representatives for the conjugacy classes of Levi subgroups of $G$. 
For $L$ in $\mathcal L$, we denote by $\Irr_\cusp (L)$ the set of isomorphism classes of 
supercuspidal irreducible representations of $L$ and we let $\Phi (L)_\cusp$ be its image in 
$\Phi_e (L)$. The group $W(G,L) = \Nor_G (L) / L$, quotient by $L$ of the normalizer of $L$ in $G$,
acts naturally on both sets. Thus we may consider the extended quotients $\Irr_\cusp (L) \q W (G,L)$ 
and $\Phi (L)_\cusp \q W (G,L)$, as well as their twisted versions. 
It follows directly from the definitions that 
\[
\big( \Irr_\cusp (L) \q W (G,L) \big)_\natural = 
\bigsqcup\nolimits_\fs \big( T_\fs \q W_\fs \big)_\natural ,
\]
where the disjoint union runs over all $\fs \in \mathfrak B (G)$ coming from supercuspidal 
$L$-representations.

\begin{thmintro}\label{thm:LLC}
The following statements hold:
\begin{itemize}
\item
If $G=\GL_m(D)$ there exists a canonical, bijective, commutative diagram
\[
\xymatrix{
\Irr (G) \ar@{<->}[rr]^{} \ar@{<->}[d] & & 
\Phi (G) \ar@{<->}[d] \\
\bigsqcup\nolimits_{L \in \mathcal L} 
\Irr_\cusp (L) \q W (G,L)  \ar@{<->}[rr] && 
\bigsqcup\nolimits_{L \in \mathcal L} \Phi (L)_\cusp \q W (G,L) .
}
\]
\item If $G=\SL_m(D)$ there exists a
family of 2-cocycles $\natural$ and a (canonical up to 
permutations within L-packets) bijective commutative diagram
\[
\xymatrix{
\Irr (G) \ar@{<->}[rr]^{} \ar@{<->}[d] & & 
\Phi_e (G) \ar@{<->}[d] \\
\bigsqcup\nolimits_{L \in \mathcal L} 
\big( \Irr_\cusp (L) \q W (G,L) \big)_{\natural}  \ar@{<->}[rr] && 
\bigsqcup\nolimits_{L \in \mathcal L} \big( \Phi (L)_\cusp \q W (G,L) \big)_{\natural}.
}
\]
\end{itemize}
\end{thmintro}

\section{Preliminaries}
\label{sec:prelim}

We start with some generalities, to fix the notations. 

Let $G$ be a connected reductive group over a local non-archimedean field $F$
of residual characteristic $p$. All our representations are
 assumed to be smooth and over the complex numbers.
We write Rep$(G)$ for the category of such $G$-representations and $\Irr (G)$ for
the collection of isomorphism classes of irreducible representations therein.
Let $P$ be a parabolic subgroup of $G$ with Levi factor $L$. The Weyl group of
$L$ is $W(G,L) = N_G (L) / L$. It acts on equivalence classes of $L$-representations
$\pi$ by
\[
(w \cdot \pi) (g) = \pi (\bar w g \bar{w}^{-1}) , 
\]
where $\bar w \in N_G (L)$ is a chosen representative for $w \in W(G,L)$. 
We write 
\[
W_\pi = \{ w \in W(G,L) \mid w \cdot \pi \cong \pi \} . 
\]
Let $\omega$ be an irreducible supercuspidal $L$-representation. 
The inertial equivalence class $\fs = [L,\omega]_G$ gives rise to a category of smooth 
$G$-representations $\Rep^\fs (G)$ and a subset $\Irr^\fs (G) \subset \Irr (G)$.
Write $X_\nr (L)$ for the group of unramified characters $L \to \C^\times$. 
Then $\Irr^\fs (G)$ consists of all irreducible  constituents of the
parabolically induced representations $I_P^G (\omega \otimes \chi)$ with
$\chi \in X_\nr (L)$. We note that $I_P^G$ always means normalized, smooth parabolic
induction from $L$ via $P$ to $G$.

The set $\Irr^{\fs_L}(L)$ with $\fs_L = [L,\omega]_L$ can be described 
explicitly, namely by 
\begin{align}
& X_{\nr} (L,\omega) = \{ \chi \in X_\nr (L) : \omega \otimes \chi \cong \omega \} , \\
& \Irr^{\fs_L} (L) = \{ \omega \otimes \chi : \chi \in X_\nr (L) / X_\nr (L,\omega) \} .
\end{align}
Several objects are attached to the Bernstein component $\Irr^\fs (G)$ of
$\Irr (G)$ \cite{BeDe}. Firstly, there is the torus                        
\[  
T_\fs := X_\nr (L) / X_\nr (L,\omega) ,
\]
which is homeomorphic to $\Irr^{\fs_L}(L)$. 
Secondly, we have the groups 
\begin{align*}
N_G (\fs_L) = & \{ g \in N_G (L) \mid g \cdot \omega \in \Irr^{\fs_L}(L) \} \\
= & \{ g \in N_G (L) \mid g \cdot [L,\omega]_L = [L,\omega]_L \} , \\
W_\fs := & \{ w \in W(G,L) \mid w \cdot \omega \in \Irr^{\fs_L}(L) \} = N_G (\fs_L) / L .
\end{align*}
Of course $T_\fs$ and $W_\fs$ are only determined up to isomorphism by $\fs$, actually
they depend on $\fs_L$. To cope with this, we tacitly assume that $\fs_L$ is known
when considering $\fs$.

The choice of $\omega \in \Irr^{\fs_L}(L)$ fixes a bijection $T_\fs \to \Irr^{\fs_L}(L)$,
and via this bijection the action of $W_\fs$ on $\Irr^{\fs_L}(L)$ is transferred to
$T_\fs$. The finite group $W_\fs$ can be thought of as the ``Weyl group" of $\fs$, 
although in general it is not generated by reflections. 

Let $C_c^\infty (G)$ be the vector space of compactly supported locally constant
functions $G \to \C$. The choice of a Haar measure on $G$ determines a convolution
product * on $C_c^\infty (G)$. The algebra $(C_c^\infty (G),*)$ is known as the 
Hecke algebra $\cH (G)$. There is an equivalence between Rep$(G)$ and the category
$\Mod (\cH (G))$ of $\cH (G)$-modules $V$ such that $\cH (G) \cdot V = V$.
We denote the collection of inertial equivalence classes for $G$ by $\mathfrak B (G)$, 
the Bernstein spectrum of $G$. The Bernstein decomposition
\[
\Rep (G) = \prod\nolimits_{\fs \in \mathfrak B (G)} \Rep^\fs (G) 
\]
induces a factorization in two-sided ideals
\[
\cH (G) = \bigoplus\nolimits_{\fs \in \mathfrak B (G)} \cH (G)^\fs .
\]
From now on we discuss things that are specific for $G = \GL_m (D)$, 
where $D$ is a central simple $F$-algebra. We write $\dim_F (D) = d^2$. Every Levi
subgroup $L$ of $G$ is conjugate to $\prod_j \GL_{\tilde m_j}(D)$ for some 
$\tilde m_j \in \N$ with $\sum_j \tilde m_j = m$. Hence every irreducible
$L$-representation $\omega$ can be written as $\otimes_j \tilde \omega_j$
with $\tilde \omega_j \in \Irr (\GL_{\tilde m_j}(D))$. Then $\omega$ is supercuspidal 
if and only if every $\tilde \omega_j$ is so. As above, we
assume that this is the case. Replacing $(L,\omega)$ by an inertially equivalent
pair allows us to make the following simplifying assumptions:

\begin{cond} \label{cond} \
\begin{itemize}
\item if $\tilde m_i = \tilde m_j$ and $[\GL_{\tilde m_j}(D),\tilde \omega_i
]_{\GL_{\tilde m_j}(D)} = [\GL_{\tilde m_j}(D),\tilde \omega_j
]_{\GL_{\tilde m_j}(D)}$, then $\tilde \omega_i = \tilde \omega_j$;
\item $\omega = \bigotimes_i \omega_i^{\otimes e_i}$, such that $\omega_i$ and $\omega_j$
are not inertially equivalent if $i \neq j$;
\item $L = \prod_i L_i^{e_i} = \prod_i \GL_{m_i}(D)^{e_i}$, embedded diagonally in 
$\GL_m (D)$ such that factors $L_i$ with the same $(m_i,e_i)$ are in subsequent positions;
\item as representatives for the elements of $W(G,L)$ we take permutation matrices;
\item $P$ is the parabolic subgroup of $G$ generated by $L$ and the upper triangular
matrices;
\item if $m_i = m_j, e_i = e_j$ and $\omega_i$ is isomorphic to $\omega_j \otimes \gamma$
for some character $\gamma$ of $\GL_{m_i}(D)$, then $\omega_i = \omega_j \otimes \gamma \chi$
for some $\chi \in X_{\nr}(\GL_{m_i}(D))$.
\end{itemize}
\end{cond}

Most of the time we will not need the conditions for stating the results, 
but they are useful in many proofs. Under Conditions \ref{cond} we consider
\begin{equation}\label{eq:1.1}
M_i = \GL_{m_i e_i}(D) \text{ naturally embedded in } 
Z_G \Big( \prod\nolimits_{j \neq i} L_j^{e_j} \Big) .
\end{equation}
Then $\prod_i M_i$ is a Levi subgroup of $G$ containing $L$. For $\fs = [L,\omega]_G$ we have
\begin{equation}\label{eq:1.2}
W_\fs = N_{\prod_i M_i} (L) / L = \prod\nolimits_i N_{M_i}(L_i^{e_i}) / L_i^{e_i} 
\cong \prod\nolimits_i S_{e_i} , 
\end{equation}
a direct product of symmetric groups.  Writing $\fs_i = [L_i,\omega_i]_{L_i}$, 
the torus associated to $\fs$ becomes
\begin{align}
\label{eq:1.4} & T_\fs = \prod\nolimits_i ( T_{\fs_i} )^{e_i} = \prod\nolimits_i T_i , \\
& T_{\fs_i} = X_\nr (L_i) / X_\nr (L_i ,\omega_i) .
\end{align}
By our choice of representatives for $W(G,L) ,\, \omega_i^{\otimes e_i}$ is stable under
$N_{M_i}(L_i^{e_i}) / L_i^{e_i} \cong S_{e_i}$. If $R_i \subset X_* (\prod\nolimits_i Z(L_i)^{e_i})$
denotes the coroot system of $(M_i,Z(L_i)^{e_i})$, we can identify $S_{e_i}$ with $W(R_i)$. 
The action of $W_\fs$ on $T_\fs$ is just permuting coordinates in the standard way and
\begin{equation}\label{eq:1.3}
W_\fs = W_\omega .
\end{equation}
The reduced norm map $D \to F$ gives rise to a group homomorphism $\Nrd : G \to F^\times$.
We denote its kernel by $G^\sharp$, so $G^\sharp$ is also the derived group of $G$.
For subgroups $H \subset G$ we write 
\[
H^\sharp = H \cap G^\sharp .
\]
In \cite{ABPS4} we determined the structure of the Hecke algebras associated to types for
$G^\sharp$, starting with those for $G$. As an intermediate step, we did this for the
group $G^\sharp Z(G)$, where $Z(G) \cong F^\times$ denotes the centre of $G$. The 
advantage is that the comparison between $G^\sharp$ and $G^\sharp Z(G)$ is easy,
while $G^\sharp Z(G) \subset G$ can be treated as an extension of finite index.
In fact it is a subgroup of finite index if $p$ does not divide $md$. In case $p$ does
divide $md$, the quotient $G / G^\sharp Z(G)$ is compact and similar techniques
can be applied.

For an inertial equivalence class $\fs = [L,\omega]_G$ we define 
$\Irr^\fs (G^\sharp)$ as the set of irreducible $G^\sharp$-representations that are 
subquotients of $\Res^G_{G^\sharp}(\pi)$ for some $\pi \in \Irr^\fs (G)$, and 
$\Rep^\fs (G^\sharp)$ as the collection of $G^\sharp$-repre\-sen\-tations 
all of whose irreducible subquotients lie in $\Irr^\fs (G^\sharp)$.
We want to investigate the category $\Rep^\fs (G^\sharp)$.
It is a product of finitely many Bernstein blocks for $G^\sharp$, see \cite[Lemma 2.2]{ABPS4}:
\begin{equation} \label{eq:procat}
\Rep^\fs (G^\sharp) = 
\prod\nolimits_{{\mf t}^\sharp\prec\fs}\Rep^{{\mf t}^\sharp} (G^\sharp).
\end{equation} 
We note that the Bernstein components $\Irr^{\mf t^\sharp}(G^\sharp)$ which are subordinate
to one $\fs$ (i.e., such that ${\mf t}^\sharp\prec\fs$) form precisely one class of 
L-indistinguishable components: every L-packet
for $G^\sharp$ which intersects one of them intersects them all.

Analogously we define $\Rep^\fs (G^\sharp Z(G))$, and we obtain
\[
\Rep^\fs (G^\sharp Z(G)) = 
\prod\nolimits_{{\mf t} \prec\fs}\Rep^{{\mf t}} (G^\sharp Z(G)) ,
\]
where the $\mf t$ are inertial equivalence classes for $G^\sharp Z(G)$.

The restriction of $\mf t$ to $G^\sharp$ is a single inertial equivalence class 
$\mf t^\sharp$, and by \cite[(43)]{ABPS4}:
\begin{equation}\label{eq:2.16}
T_{\mf t^\sharp} = T_{\mf t} / X_\nr (\Nrd (Z(G))) . 
\end{equation}
For $\pi \in \Irr (G)$ we put
\[
X^G (\pi) := \{ \gamma \in \Irr (G / G^\sharp) : \gamma \otimes \pi \cong \pi \} .
\]
The same notation will be used for representations of parabolic subgroups of
$G$ which admit a central character.
For every $\gamma \in X^G (\pi)$ there exists a nonzero intertwining operator
\begin{equation}\label{eq:2.30}
I(\gamma,\pi) \in \Hom_G (\pi \otimes \gamma ,\pi) = 
\Hom_G (\pi, \pi \otimes \gamma^{-1}),
\end{equation}
which is unique up to a scalar. As $G^\sharp \subset \ker (\gamma) ,\; I(\gamma,\pi)$
can also be considered as an element of $\End_{G^\sharp}(\pi)$. As such, 
these operators determine a 2-cocycle $\kappa_\pi$ by 
\begin{equation}\label{eq:2.21}
I(\gamma,\pi) \circ I(\gamma',\pi) = 
\kappa_\pi (\gamma,\gamma') I(\gamma \gamma' ,\pi) .
\end{equation}
By \cite[Lemma 2.4]{HiSa} they span the $G^\sharp$-intertwining algebra of $\pi$:
\begin{equation}\label{eq:2.1}
\End_{G^\sharp}\big( \mr{Res}^G_{G^\sharp} \pi \big) \cong \C [X^G (\pi) ,\kappa_\pi] ,
\end{equation}
where the right hand side denotes the twisted group algebra of $X^G (\pi)$.
Furthermore by \cite[Corollary 2.10]{HiSa}
\begin{equation}\label{eq:2.2}
\mr{Res}^G_{G^\sharp} \pi \cong \bigoplus_{\rho \in \Irr (\C [X^G (\pi) ,\kappa_\pi])}
\Hom_{\C [X^G (\pi) ,\kappa_\pi]} (\rho,\pi) \otimes \rho
\end{equation}
as representations of $G^\sharp \times X^G (\pi)$. 

The analogous groups for $\fs = [L,\omega]_G$ and $\fs_L = [L,\omega]_L$ are
\begin{align*}
& X^L (\fs) := \{ \gamma \in \Irr (L / L^\sharp Z(G)) : \gamma \otimes \omega \in
[L,\omega]_L \} , \\
& X^G (\fs) := \{ \gamma \in \Irr (G / G^\sharp Z(G)) : \gamma \otimes I_P^G (\omega) 
\in \Rep^\fs (G) \} .
\end{align*}
The role of the group $W_\fs$ for $\Rep^\fs (G^\sharp)$ is played by
\[
W_\fs^\sharp := \{ w \in W(G,L) \mid \exists \gamma \in \Irr (L / L^\sharp Z(G)) 
\text{ such that } w (\gamma \otimes \omega) \in [L,\omega]_L \}
\]
By \cite[Lemma 2.3]{ABPS4}
\begin{equation}\label{eq:1.6}
W_\fs^\sharp = W_\fs \rtimes \mf R_\fs^\sharp \text{ , where } 
\mf R_\fs^\sharp = W_\fs^\sharp \cap N_G (P \cap \prod\nolimits_i M_i) / L .
\end{equation}
while \cite[Lemma 2.4.d]{ABPS4} says that
\begin{equation}\label{eq:1.13}
X^G (\fs) / X^L (\fs) \cong \mf R_\fs^\sharp .
\end{equation}
For another way to view $X^G (\fs)$, we start with
\[
\Stab (\fs) := \{ (w,\gamma) \in N_G (L) / L \times  
\Irr (L / L^\sharp Z(G)) \mid w (\gamma \otimes \omega) \in [L,\omega]_L \} .
\] 
The normal subgroup $W_\fs$ has a complement:
\[
\begin{aligned}
& \Stab (\fs) = \Stab (\fs, P \cap \prod\nolimits_i M_i ) \ltimes W_\fs 
:= \Stab (\fs)^+ \ltimes W_\fs \\
& \Stab (\fs)^+ := \{ (w,\gamma) \in N_G (P \cap \prod\nolimits_i M_i) / L \times  
\Irr (L / L^\sharp Z(G)) \mid w (\gamma \otimes \omega) \in [L,\omega]_L \}
\end{aligned}
\]  
By \cite[Lemma 2.4.a]{ABPS4} projection of $\Stab (\fs)$ on the second
coordinate gives an isomorphism
\begin{equation}\label{eq:1.11}
X^G (\fs) \cong \Stab (\fs) / W_\fs \cong \Stab (\fs)^+ 
\end{equation}
In particular 
\begin{equation}\label{eq:1.12}
\Stab (\fs)^+ / X^L (\fs) \cong \mf R_\fs^\sharp .
\end{equation}
As in \cite[(159)--(161)]{ABPS4} we choose 
$\chi_\gamma \in X_\nr (L)^{W_\fs}$ for $(w,\gamma) \in \Stab (\fs)^+$, such that
\begin{equation}
w (\omega) \otimes \gamma \cong \omega \otimes \chi_\gamma .
\end{equation}
Notice that $\chi_\gamma$ is unique up to $X_\nr (L,\omega)$. 
Furthermore we choose an invertible
\begin{equation}\label{eq:1.22}
J(\gamma,\omega \otimes \chi_\gamma^{-1}) \in \Hom_L (\omega \otimes \chi_\gamma^{-1},
w^{-1} (\omega) \otimes \gamma^{-1}) .
\end{equation} 
This generalizes \eqref{eq:2.30} in the sense that 
\[
J(\gamma,\omega \otimes \chi_\gamma^{-1}) = I(\gamma,\omega) \quad \text{if} \quad
\gamma \in X^L (\omega) \text{ and }  \chi_\gamma = 1. 
\]
Let $V_\omega$ denote the vector space underlying $\omega$. We may assume that
\begin{equation} \label{eq:1.55}
\chi_\gamma = \gamma \quad \text{and} \quad J(\gamma,\omega \otimes \chi_\gamma^{-1}) = 
\mathrm{id}_{V_\omega} \quad \text{if} \quad \gamma \in X_{\nr}(L / L^\sharp Z(G)) .
\end{equation}

\section{Bernstein tori} 
\label{sec:tori} 

We will determine the Bernstein tori for $G^\sharp Z(G)$ and $G^\sharp$, in terms
of those for $G$. The group $X^L (\fs)$ acts on $T_\fs = \Irr^{\fs_L}(L)$ by 
$\pi \mapsto \pi \otimes \gamma$. By \cite[Proposition 2.1]{ABPS4} 
$\Res_{L^\sharp}^L (\omega)$ and $\Res_{L^\sharp}^L (\omega \otimes \chi)$ with 
$\chi \in X_{\nr}(L)$ have a common irreducible subquotient if and only if there is a 
$\gamma \in X^L (\fs)$ such that $\omega \otimes \chi \cong \omega \otimes \chi_\gamma$. 
As in \eqref{eq:1.22} we choose a nonzero
\[
J(\gamma,\omega) \in \Hom_L (\omega,\omega \otimes \chi_\gamma \gamma^{-1}) =
\Hom_L (\omega \otimes \gamma,\omega \otimes \chi_\gamma) .
\]
Then $J(\gamma,\omega) \in \Hom_{L^\sharp} (\omega,\omega \otimes \chi_\gamma)$
and for every irreducible subquotient $\sigma^\sharp$ of $\Res_{L^\sharp}^L (\omega)$
\begin{equation}\label{eq:4.15}
\gamma * (\sigma^\sharp \otimes \chi) : m \mapsto 
J(\gamma,\omega) \circ (\sigma^\sharp \otimes \chi)(m) \circ J(\gamma,\omega)^{-1} 
\end{equation}
is an irreducible subquotient representation of 
\[
\Res_{L^\sharp}^L (\omega \otimes \chi \chi_\gamma)
= \Res_{L^\sharp}^L (\omega \otimes \chi \chi_\gamma \gamma^{-1}) . 
\]
This prompts us to consider 
\begin{equation}\label{eq:2.3}
X^L (\fs,\sigma^\sharp) := \{ \gamma \in X^L (\fs) \mid \gamma * \sigma^\sharp 
\cong \sigma^\sharp \otimes \chi_\gamma \} .
\end{equation}
By \cite[Lemma 4.14]{ABPS4} 
\begin{equation}\label{eq:2.14}
\sigma^\sharp \otimes \chi \cong \sigma^\sharp \text{ for all } \chi \in X_\nr (L,\omega) .
\end{equation}
Hence the group \eqref{eq:2.3} is well-defined, that is, independent of the 
choice of the $\chi_\gamma$. For $\gamma \in X^L (\omega)$ \eqref{eq:4.15} 
reduces to $\sigma^\sharp \otimes \chi$, so $\gamma \in X^L (\fs,\sigma^\sharp)$.
By \eqref{eq:1.55} the same goes for $\gamma \in X_\nr (L / L^\sharp Z(G))$, so
there is always an inclusion
\begin{equation}\label{eq:4.2}
X^L (\omega) X_\nr (L / L^\sharp Z(G)) \subset X^L (\fs,\sigma^\sharp) .
\end{equation}
We gathered enough tools to describe the Bernstein tori for $G^\sharp$ and 
$G^\sharp Z(G)$. Recall that 
$\fs_L = [L,\omega]_L, T_\fs \cong X_\nr (L) / X_\nr (L,\omega)$.
Let $T_\fs^\sharp$ be the restriction of $T_\fs$ to $L^\sharp$, that is,
\begin{equation}\label{eq:2.39}
T_\fs^\sharp := T_\fs / X_\nr (G) = T_\fs / X_\nr (L / L^\sharp) 
\cong X_\nr (L^\sharp) / X_\nr (L,\omega) ,
\end{equation}
where $X_\nr (L / L^\sharp)$ denotes the group of unramified characters of $L$
which are trivial on $L^\sharp$.

\begin{prop}\label{prop:4.2}
Let $\sigma^\sharp$ be an irreducible subquotient of $\Res_{L^\sharp}^L (\omega)$ 
and write \\ $\mf t = [L^\sharp Z(G),\sigma^\sharp]_{G^\sharp Z(G)}$ and 
$\mf t^\sharp = [L^\sharp,\sigma^\sharp]_{G^\sharp}$. 
\enuma{
\item $X^L (\fs,\sigma^\sharp)$ depends only on $\fs_L$, not on the particular
$\sigma^\sharp$. 
\item $X_\nr (L,\omega) \{ \chi_\gamma \mid \gamma \in X^L (\fs,\sigma^\sharp) \}$ 
is a subgroup of $X_{\nr}(L)$ which contains \\
$X_{\nr}(L / L^\sharp Z(G))$.
\item $T_{\mf t} \cong T_\fs / \{ \chi_\gamma \mid \gamma \in X^L (\fs,\sigma^\sharp) \}
\cong X_{\nr}(L^\sharp Z(G)) / X_\nr (L,\omega) 
\{ \chi_\gamma \mid \gamma \in X^L (\fs,\sigma^\sharp) \}$.
\item $T_{\mf t^\sharp} \cong T_\fs^\sharp /
\{ \chi_\gamma \mid \gamma \in X^L (\fs,\sigma^\sharp) \} \cong
X_{\nr}(L^\sharp) / X_\nr (L,\omega) 
\{ \chi_\gamma \mid \gamma \in X^L (\fs,\sigma^\sharp) \}$.
}
\end{prop}
\begin{proof}
(a) By \cite[Proposition 2.1]{ABPS4} every two irreducible subquotients of \\
$\Res^L_{L^\sharp}(\omega)$ are direct summands and are conjugate by an element 
of $L$. Given $\gamma \in X^L (\fs)$, pick $m_\gamma \in L$ such that 
\[
\gamma * \sigma^\sharp \cong \big( \omega (m_\gamma)^{-1} \circ \sigma^\sharp \circ
\omega (m_\gamma) \big) \otimes \chi_\gamma = (m_\gamma \cdot \sigma^\sharp) \otimes
\chi_\gamma .
\]
For any other irreducible summand $\tau = m_\tau \cdot \sigma^\sharp$ of 
$\Res^L_{L^\sharp}(\omega)$ we compute
\begin{align*}
\gamma * \tau & = \gamma * (m_\tau \cdot \sigma^\sharp) =
J(\gamma,\omega) \circ \omega (m_\tau)^{-1} \circ \sigma^\sharp \circ \omega (m_\tau)
\circ J(\gamma,\omega)^{-1} \\
& = (\chi_\gamma \gamma^{-1} \otimes \omega)(m_\tau^{-1}) \circ J(\gamma,\omega) 
\circ \sigma^\sharp \circ J(\gamma,\omega)^{-1} \circ 
(\chi_\gamma \gamma^{-1} \otimes \omega)(m_\tau) \\
& \cong \omega (m_\tau^{-1}) \circ (m_\gamma \cdot \sigma^\sharp) \otimes 
\chi_\gamma \circ \omega (m_\tau) \\
& \cong (m_\tau m_\gamma \cdot \sigma^\sharp) \otimes \chi_\gamma .
\end{align*}
As $L / L^\sharp$ is abelian, we find that $m_\tau m_\gamma \cdot \sigma^\sharp
\cong m_\gamma m_\tau \cdot \sigma^\sharp$ and that
\[
\gamma * \tau \cong ( m_\gamma m_\tau \cdot \sigma^\sharp ) \otimes \chi_\gamma =
m_\gamma \cdot \tau \otimes \chi_\gamma .
\]
Writing $L_\tau = \{ m \in L \mid m \cdot \tau \cong \tau \}$, we deduce the
following equivalences:
\begin{multline*}
\gamma * \sigma^\sharp \cong \sigma^\sharp \otimes \chi_\gamma 
\Leftrightarrow m_\gamma \in L_{\sigma^\sharp} 
\Leftrightarrow m_\gamma \in m_\tau L_{\sigma^\sharp} m_\tau^{-1} = L_\tau
\Leftrightarrow \gamma * \tau \cong \tau \otimes \chi_\gamma .
\end{multline*}
This means that $X^L (\fs,\sigma^\sharp) = X^L (\fs,\tau)$.\\
(b) By \eqref{eq:1.55} and \eqref{eq:4.2}
\[
X_{\nr}(L / L^\sharp Z(G)) \subset
\{ \chi_\gamma \mid \gamma \in X^L (\fs,\sigma^\sharp) \}.
\]
In view of the uniqueness property of $\chi_\gamma$ the map
\[
X^L (\fs) \to X_{\nr}(L) / X_{\nr}(L,\omega) : \gamma \mapsto \chi_\gamma
\]
is a group homomorphism with kernel $X^L (\omega)$. Hence the
$\chi_\gamma$ form a subgroup of $X_{\nr}(L) / X_\nr (L,\omega)$, 
isomorphic to $X^L (\fs) / X^L (\omega)$.\\
(c) Consider the family of $L^\sharp Z(G)$-representations
\[
\{ \sigma^\sharp \otimes \chi \mid \chi \in X_\nr (L) \} . 
\]
We have to determine the $\chi$ for which $\sigma^\sharp \otimes \chi \cong 
\sigma^\sharp \in \Irr (L^\sharp Z(G))$. From \cite[Lemma 4.14]{ABPS4} we see that
this includes all the elements of $X_\nr (L,\omega) X_\nr (L / L^\sharp Z(G))$.
By \cite[Proposition 2.1.b]{ABPS4} and part (a), all the remaining $\chi$ come
from $\{\chi_\gamma \mid \gamma \in X^L (\fs,\sigma^\sharp) \}$. This 
gives the first isomorphism, and the second follows with part~(b). \\
(d) This is a consequence of part (c) and \eqref{eq:2.16}.
\end{proof}

Proposition \ref{prop:4.2} entails that for every inertial equivalence class
\begin{equation*}
\mf t = [L^\sharp Z(G),\sigma^\sharp]_{G^\sharp Z(G)} \; \prec \;
\fs = [L,\omega]_G 
\end{equation*}
the action \eqref{eq:4.15} of $X^L (\fs,\sigma^\sharp)$ leads to
\[
T_{\mf t} \cong T_\fs / X^L (\fs, \sigma^\sharp) .
\]
However, some of the tori 
\[
T_{\mf t} = T_{\mf t_L} = 
\Irr^{[L^\sharp Z(G),\sigma^\sharp]_{L^\sharp Z(G)}} (L^\sharp Z(G))
\]
associated to inequivalent $\sigma^\sharp \subset \Res_{L^\sharp}^L (\omega)$ 
can coincide as subsets of $\Irr (L^\sharp Z(G))$. This is caused by elements of 
$X^L (\fs) \setminus X^L (\fs, \sigma^\sharp)$ via the action \eqref{eq:4.15}. 
With \eqref{eq:2.14}, \eqref{eq:2.3} and \eqref{eq:2.2} we can write
\begin{equation}\label{eq:4.11}
\Irr^{\fs_L} (L^\sharp Z(G)) = 
\bigcup\nolimits_{\mf t_L \prec \fs_L} T_{\mf t_L} =
\Big( T_\fs \times \Irr (\C [X^L (\omega),\kappa_\omega]) \Big) / X^L (\fs) ,
\end{equation}
where $(\omega \otimes \chi, \rho) \in T_\fs \times 
\Irr (\C [X^L (\omega),\kappa_\omega])$ corresponds to 
\[
\Hom_{\C [X^L (\omega),\kappa_\omega]}(\rho,\omega \otimes \chi)
\; \in \; \Irr (L^\sharp Z(G)).
\]
With \eqref{eq:2.16} we can deduce a similar expression for $L^\sharp$:
\begin{equation}\label{eq:4.4}
\begin{aligned}
\Irr^{\fs_L} (L^\sharp) = 
\bigcup\nolimits_{\mf t_L^\sharp \prec \fs_L} T_{\mf t_L^\sharp} & = 
\Big( T_\fs^\sharp \times \Irr (\C [X^L (\omega),\kappa_\omega]) \Big) / X^L (\fs) \\ 
& = \Big( T_\fs \times \Irr (\C [X^L (\omega),\kappa_\omega]) \Big) / 
X^L (\fs) X_\nr (L^\sharp Z(G) / L^\sharp) .
\end{aligned}
\end{equation}
In the notation of \eqref{eq:4.11} and \eqref{eq:4.4} the action of 
$\gamma \in X^L (\fs)$ becomes
\begin{equation}\label{eq:4.3}
\gamma \cdot (\omega \otimes \chi,\rho) = 
(\omega \otimes \chi \chi_\gamma, \phi_{\omega,\gamma} \rho) ,
\end{equation}
where $\phi_{\omega,\gamma}$ is yet to be determined. Any $\gamma \in X^L (\omega)$ 
can be adjusted by an element of $X_\nr (L,\omega)$ to achieve $\chi_\gamma = 1$. 
Then \eqref{eq:2.14} shows that $\phi_{\omega,\gamma} \rho \cong \rho$ 
for all $\gamma \in X^L (\omega)$.

\begin{lem}\label{lem:4.1}
For $\gamma \in X^L (\fs) ,\; \phi_{\omega,\gamma} \rho$ is $\rho$ 
tensored with a character of 
$X^L (\omega)$, which we also call $\phi_{\omega,\gamma}$. Then 
\[
X^L (\fs) \to \Irr (X^L (\omega)) : \gamma \mapsto \phi_{\omega,\gamma}
\]
is a group homomorphism. 
\end{lem}
\begin{proof}
Let $N_{\gamma'}$ be a standard basis element of $\C[X^L (\omega),\kappa_\omega]$. 
In view of \eqref{eq:4.15} $\phi_{\omega,\gamma} \rho$ is given by
\begin{equation}\label{eq:2.5}
N_{\gamma'} \mapsto J(\gamma,\omega) I(\gamma',\omega) J(\gamma,\omega)^{-1} \in
\Hom_L (\omega \otimes \gamma' \chi_\gamma, \omega \otimes \chi_\gamma) .
\end{equation}
Since these are irreducible $L$-representations, 
there is a unique $\lambda \in \C^\times$ such that 
\begin{align*}
& J(\gamma,\omega) I(\gamma',\omega) J(\gamma,\omega)^{-1} = 
\lambda^{-1} I(\gamma',\omega \otimes \chi_\gamma) , \\
& (\phi_{\omega,\gamma} \rho)(N_{\gamma'}) = \rho (\lambda I (\gamma',\omega)) =
\lambda \rho (N_{\gamma'}).
\end{align*}
Moreover the relation
\begin{equation}\label{eq:4.5}
I (\gamma'_1,\omega \otimes \chi_\gamma) I (\gamma'_2,\omega \otimes \chi_\gamma) =
\kappa_{\omega \otimes \chi_\gamma}(\gamma'_1,\gamma'_2) 
I(\gamma'_1 \gamma'_2, \omega \otimes \chi_\gamma) 
\end{equation}
also holds with $J(\gamma,\omega) I(\gamma'_i,\omega) J(\gamma,\omega)^{-1}$ 
instead of $I(\gamma'_i,\omega)$-- a basic property of conjugation. 
It follows that $\gamma' \mapsto \lambda$ defines a character of $X^L (\omega)$ 
which implements the action $\rho \mapsto \phi_{\omega,\gamma} \rho$. 
As $\phi_{\omega,\gamma}$ comes from conjugation by 
$J(\gamma,\omega \otimes \chi)$ and by \eqref{eq:4.5}, 
$\gamma \mapsto \phi_{\omega,\gamma}$ is a group homomorphism. 
\end{proof}

A straightforward check, using the above proof, shows that
\begin{equation}\label{eq:2.4}
\begin{array}{ccc}
\Hom_{\C [X^L (\omega),\kappa_\omega]}(\rho,\omega \otimes \chi) & \to &
\Hom_{\C [X^L (\omega),\kappa_\omega]}
(\phi_{\omega,\gamma} \rho,\omega \otimes \chi \chi_\gamma) \\
f & \mapsto & J(\gamma,\omega \otimes \chi) \circ f
\end{array}
\end{equation}
is an isomorphism of $L^\sharp Z(G)$-representations.

\section{Hecke algebras}

We will show that the algebras $\mc H (G^\sharp Z(G))^\fs$ and $\mc H (G^\sharp )^\fs$
are stratified equivalent \cite{ABPS8} with much simpler algebras. In this section we recall
the final results of \cite{ABPS4}, which show that up to Morita equivalence these algebras 
are closely related to affine Hecke algebras. In section \ref{sec:geomEquiv}
we analyse the latter algebras in the framework of \cite{ABPS8}.

Our basic affine Hecke algebra is called $\cH (T_\fs,W_\fs,q_\fs)$. 
By definition \cite[(119)]{ABPS4} it has a $\C$-basis
$\{ \theta_x [w] : x \in X^* (T_\fs), w \in W_\fs \}$ such that 
\begin{itemize}
\item the span of the $\theta_x$ is identified with the algebra $\mc O (T_\fs)$ 
of regular functions on $T_\fs$;
\item the span of the $[w]$ is the finite dimensional Iwahori--Hecke algebra
$\cH (W_\fs,q_\fs)$;
\item the multiplication between these two subalgebras is given by 
\begin{equation}\label{eq:1.14}
f [s] - [s] (s \cdot f) = (q_\fs (s) - 1)(f - (s \cdot f)) (1 - \theta_{-\alpha})^{-1}
\qquad f \in \mc O (T_\fs) ,
\end{equation}
for a simple reflection $s = s_\alpha$;
\item the algebra is well-defined for any array of parameters $q_\fs = (q_{\fs,i})_i$ in 
$\C^\times$. The parameters $q_{\fs,i}$ that we will use are given explicitly in 
\cite[Th\'eor\`eme 4.6]{Sec3}.
\end{itemize}
Thus $\cH (T_\fs,W_\fs,q_\fs)$ is a tensor product of affine Hecke algebras of type
$GL_e$, but written in such a way that the torus $T_\fs$ appears canonically in it
(i.e. independent of the choice of a base point of $T_\fs$).
S\'echerre and Stevens \cite{Sec3,SeSt6} showed that 
\begin{equation}\label{eq:1.30}
\mc H (G)^\fs \text{ is Morita equivalent with } \mc H (T_\fs,W_\fs,q_\fs).
\end{equation}
From \cite{SeSt4} we know that there exists a simple type $(K,\lambda)$ for 
$[L,\omega]_M$, and in \cite{SeSt6} it was shown to admit a $G$-cover $(K_G,\lambda_G)$. 
We denote the associated central idempotent of $\cH (K)$ by $e_\lambda$, and similarly 
for other irreducible representations. Then $V_\lambda = e_\lambda V_\omega$.

For the restriction process we need an idempotent that is invariant under 
$X^G (\fs)$. To that end we replace $\lambda_G$ by the sum of the representations 
$\gamma \otimes \lambda_G$ with $\gamma \in X^G (\fs)$, which we call $\mu_G$.
In \cite[(91)]{ABPS4} we constructed an idempotent $e_{\mu_G} \in \cH (G)$ which is supported
on the compact open subgroup $K_G \subset G$. It follows from the work of S\'echerre
and Stevens \cite{SeSt6} that $e_{\mu_G} \cH (G) e_{\mu_G}$ is Morita equivalent
with $\cH (G)^\fs$. 

In \cite[(128)]{ABPS4} we defined a finite dimensional subspace
\[
V_\mu := \sum\nolimits_{\gamma \in X^G (\fs)} e_{\gamma \otimes \lambda} V_\omega
\]
of $V_\omega$ which is stable under the operators $I(\gamma,\omega)$ with 
$\gamma \in X^L (\fs)$. By \cite{Sec3} and \cite[Theorem 4.5.d]{ABPS4}
\begin{equation}\label{eq:1.7}
e_{\mu_G} \cH (G) e_{\mu_G} \cong \cH (T_\fs,W_\fs,q_\fs) \otimes 
\End_\C ( V_\mu \otimes_\C \C \mf R_\fs^\sharp) .
\end{equation}
The groups $X^G (\fs)$ and $X_\nr (G)$ act on $e_{\mu_G} \cH (G) e_{\mu_G}$ by 
pointwise multiplication of functions $G \to \C$ with characters of $G$. However, for
technical reasons we use the action
\begin{equation}\label{eq:1.9}
\alpha_\gamma (f) (g) = \gamma^{-1}(g) f(g) \qquad f \in \cH (G), 
\gamma \in \Irr (G / G^\sharp), g \in G .
\end{equation}
The action on the right hand side of \eqref{eq:1.7} preserves the tensor factors, 
and on $\End_\C (\C \mf R_\fs^\sharp)$ it is the natural action of 
$X^G (\fs) / X^L (\fs) \cong \mf R_\fs^\sharp$.

Although $e_{\mu_G}$ looks like the idempotent of a type, it is not clear whether
it is one, because the associated $K_G$-representation is reducible and no more
suitable compact subgroup of $G$ is in sight. Let $e_{\mu_{G^\sharp}}$ (respectively
$e_{\mu_{G^\sharp Z(G)}}$) be the restriction of $e_{\mu_G} : G \to \C$ to $G^\sharp$
(resp. $G^\sharp Z(G)$). We normalize the Haar measure on $G^\sharp$ (resp.
$G^\sharp Z(G)$) such that it becomes an idempotent in $\cH (G^\sharp)$ (resp.
$\cH (G^\sharp Z(G))$).

In \cite[Lemma 3.3]{ABPS4} we  constructed a certain finite set
$[L / H_\lambda]$, consisting of representatives for a normal subgroup 
$H_\lambda \subset L$. Consider the elements 
\begin{equation}\label{eq:1.25}
\begin{aligned}
& e^\sharp_{\lambda_G} := \sum\nolimits_{a \in [L / H_\lambda]} a e_{\mu_G} a^{-1} \in \cH (G), \\ 
& e^\sharp_{\lambda_{G^\sharp Z(G)}} := \sum\nolimits_{a \in [L / H_\lambda]} 
a e_{\mu_{G^\sharp Z(G)}} a^{-1} \in \cH (G^\sharp Z(G)), \\ 
& e^\sharp_{\lambda_{G^\sharp}} := \sum\nolimits_{a \in [L / H_\lambda]} 
a e_{\mu_{G^\sharp}} a^{-1} \in \cH (G^\sharp).
\end{aligned}
\end{equation}
It follows from \cite[Lemma 3.12]{ABPS4} that they are again idempotent. Notice that
$e^\sharp_{\lambda_G}$ detects the same category of $G$-representations as $e_{\mu_G}$,
namely $\Rep^\fs (G)$. In the proof of \cite[Proposition 3.15]{ABPS4} we established
that \eqref{eq:1.7} extends to an isomorphism
\begin{equation}\label{eq:1.10}
e^\sharp_{\lambda_G} \cH (G) e^\sharp_{\lambda_G} \cong \cH (T_\fs,W_\fs,q_\fs) \otimes 
\End_\C ( V_\mu \otimes_\C \C \mf R_\fs^\sharp) \otimes M_{[L : H_\lambda]}(\C) .
\end{equation}

\begin{thm}\label{thm:1.1}
\textup{\cite[Theorem 4.13]{ABPS4}} Let $G = \GL_m(D)$ be an inner form of $\GL_n(F)$.   
Then for any $\fs \in \fB(G)$:
\enuma{
\item $\cH (G^\sharp Z(G))^\fs$ is Morita equivalent with its subalgebra
\[
e^\sharp_{\lambda_{G^\sharp Z(G)}} \cH (G^\sharp Z(G)) e^\sharp_{\lambda_{G^\sharp Z(G)}}
= \bigoplus\nolimits_{a \in [L / H_\lambda]} a e_{\mu_{G^\sharp Z(G)}} a^{-1} 
\cH (G^\sharp Z(G)) a e_{\mu_{G^\sharp Z(G)}} a^{-1}
\]
\item Each of the algebras $a e_{\mu_{G^\sharp Z(G)}} a^{-1} \cH (G^\sharp Z(G))
a e_{\mu_{G^\sharp Z(G)}} a^{-1}$ is isomorphic to 
\begin{equation}\label{eq:1.5}
\Big( \cH (T_\fs,W_\fs,q_\fs) \otimes \End_\C ( V_\mu ) \Big)^{X^L (\fs)} 
\rtimes \mf R_\fs^\sharp . 
\end{equation}
\item Under these isomorphisms the action of $X_\nr (G)$ on $\cH (G^\sharp Z(G))^\fs$
becomes the action of $X_\nr (L / L^\sharp) \cong X_\nr (G)$ on \eqref{eq:1.5}
via translations on $T_\fs$.
}
\end{thm}

Recall from \eqref{eq:2.39} that $T_\fs^\sharp$ is the restriction of $T_\fs$ to 
$L^\sharp$. With this torus we build an affine Hecke algebra 
$\cH (T_\fs^\sharp, W_\fs, q_\fs)$ for $G^\sharp$.

\begin{thm}\label{thm:1.2}
\textup{\cite[Theorem 4.15]{ABPS4}}
\enuma{
\item $\cH (G^\sharp)^\fs$ is Morita equivalent with 
\[
e^\sharp_{\lambda_{G^\sharp}} \cH (G^\sharp) e^\sharp_{\lambda_{G^\sharp}} =
\bigoplus\nolimits_{a \in [L / H_\lambda]} a e_{\mu_{G^\sharp}} a^{-1} 
\cH (G^\sharp) a e_{\mu_{G^\sharp}} a^{-1}
\]
\item Each of the algebras $a e_{\mu_{G^\sharp}} a^{-1} \cH (G^\sharp)
a e_{\mu_{G^\sharp}} a^{-1}$ is isomorphic to 
\[
\Big( \cH (T_\fs^\sharp,W_\fs,q_\fs) \otimes \End_\C ( V_\mu ) 
\Big)^{X^L (\fs)} \rtimes \mf R_\fs^\sharp . 
\]
}
\end{thm}

Let us describe the above actions of the group $X^G (\fs)$ explicitly. 
The action on
\begin{equation}\label{eq:1.8}
a e_{\mu_{G^\sharp Z(G)}} a^{-1} \cH (G^\sharp Z(G)) a e_{\mu_{G^\sharp Z(G)}} 
a^{-1} \cong \cH (T_\fs, W_\fs, q_\fs) \otimes \End_\C (V_\mu ) .
\end{equation}
does not depend on $a \in [L /H_\lambda]$ because 
\[
\alpha_\gamma (a f a^{-1}) = a (\alpha_\gamma (f)) a^{-1} \qquad f \in \cH (G) .
\]
The isomorphism \eqref{eq:1.12} yields an action $\alpha$ of $\Stab (\fs)^+$ on \eqref{eq:1.8}.

\begin{thm}\label{thm:1.3}
\textup{\cite[Lemmas 3.5 and 4.11]{ABPS4}}
\enuma{
\item The action of $\Stab (\fs)^+$ on $\cH (T_\fs, W_\fs, q_\fs) \otimes 
\End_\C (V_\mu)$ in Theorem \ref{thm:1.1}
preserves both tensor factors. On $\cH (T_\fs, W_\fs, q_\fs)$ it is given by
\[
\alpha_{(w,\gamma)}(\theta_x [v]) = \chi_\gamma^{-1}(x) \theta_{w(x)} [w v w^{-1}] 
\qquad x \in X^* (T_\fs), v \in W_\fs ,
\]
and on $\End_\C (V_\mu)$ by
\[
\alpha_{(w,\gamma)}(h) = J(\gamma,\omega \otimes \chi_\gamma^{-1}) \circ h \circ
J(\gamma,\omega \otimes \chi_\gamma^{-1})^{-1} .
\]
\item The subgroup of elements that act trivially is 
\[
X^L (\omega,V_\mu) = \big\{ \gamma \in X^L (\omega) \mid 
I(\gamma,\omega) |_{V_{\mu }} \in \C^\times \mathrm{id}_{V_{\mu }} \big\} .
\]
Its cardinality equals $[L : H_\lambda]$.
\item Part (a) and Theorem \ref{thm:1.1}.c also describe the action of 
$\Stab (\fs)^+ X_\nr (G)$ on $\cH (T_\fs^\sharp, W_\fs, q_\fs) \otimes 
\End_\C (V_\mu)$ in Theorem \ref{thm:1.2}. The subgroup of elements that act 
trivially on this algebra is 
\[
X^L (\omega,V_\mu) X_\nr (G ) = X^L (\omega,V_\mu) X_\nr ( L / L^\sharp) .
\]
}
\end{thm}

\section{Spectrum preserving morphisms and stratified equivalences} 
\label{sec:geomEquiv}

We will show that the Hecke algebras obtained in Theorems \ref{thm:1.1} and \ref{thm:1.2}
fit in the framework of spectrum preserving morphisms and stratified equivalence 
of finite type algebras, see \cite{ABPS8}. First we exhibit an 
algebra that interpolates between 
\[
\big( \cH (T_\fs,W_\fs,q_\fs) \otimes \End_\C (V_\mu) \big)^{X^L (\fs)}
\rtimes \mf R_\fs^\sharp
\]
and $\big( \mc O (T_\fs) \rtimes W_\fs \otimes \End_\C (V_\mu) \big)^{X^L (\fs)} 
\rtimes \mf R_\fs^\sharp$. Recall that Conditions \ref{cond} are in force and write 
\[
T_\fs = \prod\nolimits_i T_i ,\; R_\fs = \bigsqcup\nolimits_i R_i ,\; 
W_\fs = \prod\nolimits_i W (R_i) = \prod\nolimits_i S_{e_i} .
\]
Let $q_i$ be the restriction of $q_\fs : X^* (T_\fs) \rtimes W_\fs \to \R_{>0}$
to $X^* (T_i) \rtimes W(R_i)$. Recall Lusztig's asymptotic Hecke algebra 
$J(X^* (T_i) \rtimes W(R_i))$ from 
\cite{Lus3,Lus4}. We remark that, although in \cite{Lus3} it is supposed that the 
underlying root datum is semisimple, this assumption is shown to be unneccesary
in \cite{Lus4}. This algebra is unital and of finite type over 
$\mc O (T_i)^{W(R_i)}$. It has a distinguished $\C$-basis 
$\{ t_{xv} \mid x \in X^* (T_i), v \in W(R_i) \}$ and the $t_x$ with 
$x \in X^* (T_i)^{W(R_i)}$ are central. We define
\[
J (X^* (T_\fs) \rtimes W_\fs) = \bigotimes\nolimits_i J(X^* (T_i) \rtimes W(R_i)) .
\]
This is a unital finite type algebra over $\mc O (T_\fs)^{W_\fs}$, in fact for
several different $\mc O (T_\fs)^{W_\fs}$-module structures.

Lusztig \cite[\S 1.4]{Lus4} defined injective algebra homomorphisms
\begin{equation}\label{eq:4.9}
\cH (T_i, W(R_i),q_i) \xrightarrow{\phi_{i,q_i}} J (X^* (T_i) \rtimes W(R_i))
\xleftarrow{\phi_{i,1}} \mc O (T_i) \rtimes W(R_i)  
\end{equation}
with many nice properties. Among these, we record that 
\begin{equation}\label{eq:5.1}
\phi_{i,q_i} \text{ and } \phi_{i,1} \text{ are the identity on }
\C [X^* (T_i)^{W(R_i)}] \cong \mc O \big( X_\nr (Z(M_i)) \big) . 
\end{equation}
There exist $\mc O (T_i)^{W(R_i)}$-module structures on $J (X^* (T_i) \rtimes W(R_i))$
for which the maps \eqref{eq:4.9} are $\mc O (T_i)^{W(R_i)}$-linear, namely by letting
$\mc O (T_i)^{W(R_i)}$ act via the map $\phi_{i,q_i}$ or via $\phi_{i,1}$.
Taking tensor products over $i$ in \eqref{eq:4.9} and with the identity on 
$\End_\C (V_{\mu})$ gives injective algebra homomorphisms
\begin{equation}\label{eq:3.3}
\begin{array}{rrcc}
\phi_{q_\fs} : & \cH (T_\fs,W_\fs,q_\fs) \otimes \End_\C (V_\mu) & \to &
J (X (T_\fs) \rtimes W_\fs) \otimes \End_\C (V_\mu) , \\
\phi_1 : & \mc O (T_\fs) \rtimes W_\fs \otimes \End_\C (V_\mu) & \to & 
J (X (T_\fs) \rtimes W_\fs) \otimes \End_\C (V_\mu) .
\end{array}
\end{equation}
The maps $\phi_{q_\fs}$ and $\phi_1$ are $\mc O (T_\fs)^{W_\fs}$-linear with 
respect to the appropriate module structure on $J (X (T_\fs) \rtimes W_\fs)$.

\begin{lem}\label{lem:4.3}
Via \eqref{eq:3.3} the action of $\Stab (\fs)^+$ on 
$\cH (T_\fs,W_\fs,q_\fs) \otimes \End_\C (V_\mu)$ from Theorem \ref{thm:1.3}
extends canonically to an action on 
$J (X^* (T_\fs) \rtimes W_\fs) \otimes \End_\C (V_\mu)$, which stabilizes the 
subalgebra $\mc O (T_\fs) \rtimes W_\fs \otimes \End_\C (V_\mu)$.
\end{lem}
\begin{proof}
In $X^* (T_\fs) \rtimes W(G,L)$ every $w \in \mf R_\fs^\sharp$ normalizes
the subgroup $X^* (T_\fs) \rtimes W_\fs$. The group automorphism
\begin{equation}\label{eq:4.6}
xv \mapsto w xv w^{-1}  \quad \text{of} \quad X^* (T_\fs) \rtimes W_\fs 
\end{equation}
only permutes the subgroups $X^* (T_i) \rtimes W(R_i)$. In particular it
stabilizes the set of simple reflections. With the canonical
representatives for $W_\fs^\sharp$ in $G$ from \eqref{eq:1.6}, conjugation by $w$ 
stabilizes the type for $\fs_L$ (it is a product of simple
types in the sense of \cite[\S 4]{Sec3}). Hence \eqref{eq:4.6} preserves $q_\fs$. 
Thus \eqref{eq:4.6} can be factorized as
\begin{equation}\label{eq:4.40}
\prod_j \omega_j \text{ with } \omega_j \in \mathrm{Aut} 
\Big( \prod_{i : R_i = R_j, q_i = q_j} X^* (T_i) \rtimes W(R_i) \Big)
\end{equation}
The function $q_\fs$ takes the same value on all simple (affine) roots 
associated to the group for one $j$ in \eqref{eq:4.40}, so the algebra 
\begin{equation}\label{eq:4.7}
\bigotimes\nolimits_{i \mid R_i = R_j, q_i = q_j} J (X^* (T_i) \rtimes W(R_i)) 
\end{equation}
is of the kind considered in \cite[\S 1]{Lus4}. Then $\omega_j$ is an 
automorphism which fits in a group called $\Omega$ in \cite[\S 1.1]{Lus4}, 
so it gives rise to an automorphism of the algebra \eqref{eq:4.7}. In this way 
the group $\mf R_\fs^\sharp \cong \mr{Stab} (\fs)^+ / X^L (\fs)$ 
acts naturally on $J(X^* (T_\fs) \rtimes W_\fs)$. 

Since $T_\fs^{W_\fs}$ is central in $T_\fs \rtimes W_\fs$, every 
$\chi \in T_\fs^{W_\fs}$ gives rise to an algebra automorphism of 
$J(X^* (T_\fs) \rtimes W_\fs)$:
\begin{equation}\label{eq:4.8}
t_{xv} \mapsto \chi (x) t_{xv} \qquad x \in X^* (T_\fs), v \in W_\fs . 
\end{equation}
Thus we can make $\Stab (\fs)^+$ act on $J(X^* (T_\fs) \rtimes W_\fs)$ by
\[
(w,\gamma) \cdot t_{xv} = \chi_\gamma^{-1}(x) t_{w xv w^{-1}} 
\qquad x \in X^* (T_\fs), v \in W_\fs .
\]
The action of $\Stab (\fs)^+$ on $\End_\C (V_{\mu})$ may be
copied to this setting, so we can define the following action on 
$J (X^* (T_\fs) \rtimes W_\fs) \otimes \End_\C (V_{\mu})$:
\[
\alpha_{(w,\gamma)}(t_{xv} \otimes h) =  \chi_\gamma^{-1}(x) t_{w xv w^{-1}} 
\otimes  J(\gamma,\omega \otimes \chi_\gamma^{-1}) \circ h \circ
J(\gamma,\omega \otimes \chi_\gamma^{-1})^{-1} .
\]
Of course the above also works with the label function 1 instead of $q_\fs$. 
That yields a similar action of $\Stab (\fs)^+$ on 
$\mc O (T_\fs) \rtimes W_\fs \otimes \End_\C (V_{\mu})$, namely
\begin{equation}\label{eq:4.14}
\alpha_{(w,\gamma)}(xv \otimes h) =  \chi_\gamma^{-1}(x) \; w xv w^{-1}
\otimes  J(\gamma,\omega \otimes \chi_\gamma^{-1}) \circ h \circ
J(\gamma,\omega \otimes \chi_\gamma^{-1})^{-1} ,
\end{equation}
where $xv \in X^* (T_\fs) \rtimes W_\fs$. It follows from \cite[\S 1.4]{Lus4}
that $\phi_{q_\fs}$ and $\phi_1$ are now $\mr{Stab}(\fs)^+$-equivariant.
\end{proof}

\begin{lem}\label{lem:3.4}
The $\mc O (T_\fs)^{W_\fs}$-algebra homomorphisms $\phi_{q_\fs}$ and $\phi_1$
from \eqref{eq:3.3} are spectrum preserving with respect to filtrations, 
in the sense of \cite{ABPS8}. 
\end{lem}
\begin{proof}
It suffices to consider the map $\phi_{q_\fs}$, for the same
reasoning will apply to $\phi_1$.
Our argument is a generalization of \cite[Theorem 10]{BaNi}, which proves
the analogous statements for $J (X^* (T_i) \rtimes W(R_i))$. Recall the function
\begin{equation}\label{eq:4.13}
a : X^* (T_\fs) \rtimes W_\fs \to \Z_{\geq 0}
\end{equation}
from \cite[\S 1.3]{Lus4}. For fixed $n \in \Z_{\geq 0}$, the subspace of
$J(X^* (T_i) \rtimes W(R_i))$ spanned by the $t_{xv}$ with $a(xv) = n$ is a 
two-sided ideal, let us call it $J^{i,n}$. Then
\[
J(X^* (T_i) \rtimes W(R_i)) = \bigoplus\nolimits_{n \geq 0} J^{i,n}  
\]
and the sum is finite by \cite[\S 7]{Lus1}. Moreover 
\[
\cH^{i,n} := \phi_{i,q_i}^{-1} \Big( \bigoplus\nolimits_{k \geq n} J^{i,k} \Big) 
\]
is a two-sided ideal of $\cH (T_i,W(R_i),q_i)$. According to 
\cite[Corollary 3.6]{Lus3} the morphism of $\mc O (T_i)^{W(R_i)}$-algebras
\[
\cH^{i,n} / \cH^{i,n+1} \to J^{i,n} \text{ induced by } \phi_{i,q_i}
\] 
is spectrum preserving. For any irreducible $J^{i,n}$-module $M^i_J$ the 
$\cH^{i,n}$-module $\phi_{q_i,i}^* (M^i_J)$ has a distinguished quotient 
$M^i_\cH$, which is an irreducible $\cH^{i,n} / \cH^{i,n+1}$-module. 

Let $\mathbf n$ be a vector with coordinates $n_i \in \Z_{\geq 0}$ and put
$|\mathbf n| = \sum_i n_i$. We write $\mathbf n \leq \mathbf n'$ if $n_i \leq
n'_i$ for all $i$. We define the two-sided ideals
\[
\begin{array}{lllll}
J^{\mathbf n} & = & \bigotimes_i J^{i,n_i} \otimes \End_\C (V_\mu)
& \subset & J (X^* (T_\fs) \rtimes W_\fs) \otimes \End_\C (V_\mu) ,\\
\cH^{\mathbf n} & = & \bigotimes_i \cH^{i,n_i} \otimes \End_\C (V_\mu)
& \subset & \cH (T_\fs,W_\fs,q_\fs) \otimes \End_\C (V_\mu) ,\\
\cH^{\mathbf n +} & = & 
\sum_{\mathbf n' \geq \mathbf n, |\mathbf n'| = |\mathbf n| + 1} \cH^{\mathbf n'} .
\end{array}
\]
It follows from the above that the morphism of $\mc O (T_\fs)^{W_\fs}$-algebras
\begin{equation}\label{eq:4.12}
\bigotimes\nolimits_i ( \cH^{i,n_i} / \cH^{i,n_i + 1} ) 
\otimes \End_\C (V_{\mu})  
\cong \cH^{\mathbf n} / \cH^{\mathbf n +} \to J^{\mathbf n}
\end{equation}
induced by $\phi_{q_\fs}$ is spectrum preserving, and that every irreducible
$J^{\mathbf n}$-module $M_J$ has a distinguished quotient $M_\cH$ which is an 
irreducible $\cH^{\mathbf n} / \cH^{\mathbf n +}$-module.

Next we define, for $n \in \Z_{\geq 0}$: 
\begin{equation}\label{eq:4.41}
J^n := \bigoplus_{|\mathbf n| = n} J^{\mathbf n} ,\; 
\cH^n := \bigoplus_{|\mathbf n| = n} \cH^{\mathbf n} .
\end{equation}
The aforementioned properties of the map \eqref{eq:4.12} are also valid for
\begin{equation}\label{eq:4.17}
\cH^n / \cH^{n+1} \to J^n ,
\end{equation} 
which shows that $\phi_{q_\fs}$ is spectrum preserving with respect to the 
filtrations $(\cH^n )_{n \geq 0}$ and $(\oplus_{m \geq n} J^m )_{n \geq 0}$.
\end{proof}

By Lemma \ref{lem:4.3} and \eqref{eq:1.12} 
\begin{equation}\label{eq:4.10}
\begin{array}{l}
\Big( \cH (T_\fs,W_\fs,q_\fs) \otimes \End_\C (V_\mu) \Big)^{X^L (\fs)} 
\rtimes \mf R_\fs^\sharp , \\
\Big( J (X^* (T_\fs) \rtimes W_\fs) \otimes \End_\C (V_\mu) 
\Big)^{X^L (\fs)} \rtimes \mf R_\fs^\sharp , \\
\Big( \mc O (T_\fs) \rtimes W_\fs \otimes \End_\C (V_\mu) 
\Big)^{X^L (\fs)} \rtimes \mf R_\fs^\sharp
\end{array}
\end{equation}
are unital finite type $\mc O (T_\fs)^{\Stab (\fs)}$-algebras, while 
$\phi_{q_\fs}$ and $\phi_1$ provide morphisms between them. 

\begin{thm}\label{thm:4.4} 
\enuma{
\item The above morphisms between the $\mc O (T_\fs)^{\Stab (\fs)}$-algebras
\eqref{eq:4.10} are spectrum preserving with respect to filtrations. 
\item The same holds for the three algebras of \eqref{eq:4.10}
with $T_\fs^\sharp$ instead of $T_\fs$.
}
\end{thm}
\begin{proof}
(a) We use the notations from the proof of Lemma \ref{lem:3.4}.
Since Lusztig's $a$-function is constant on two-sided cells \cite[\S 1.3]{Lus4}
and conjugation by elements of $\mf R_\fs^\sharp$ preserves the set of simple 
(affine) reflections in $X^* (T_\fs) \rtimes W_\fs$:
\[
a(w xv w^{-1}) = a (xv) \text{ for all } 
x \in X^* (T_\fs), v \in W_\fs, w \in \mf R_\fs^\sharp .
\]
Hence $J^n$ and $\cH^n$ are stable under the respective actions $\alpha$ and
\eqref{eq:4.17} is $\Stab (\fs)^+$-equivariant. Let $M_J$ be an irreducible 
$J^{\mb n}$-module and regard it as a $(J^n)^{X^L (\fs)}$-module via 
the map $J^n \to J^{\mb n}$ from \eqref{eq:4.41}. By Clifford theory (see 
\cite[Appendix]{RaRa}) its decomposition is governed by a twisted group algebra
of the stabilizer of $M_J$ in $X^L (\fs)$. Since \eqref{eq:4.17} is 
$X^L (\fs)$-equivariant and $M_\cH$ is a quotient of $M_J$, the decomposition
of $M_\cH$ as module over $(\cH^n / \cH^{n+1})^{X^L (\fs)}$ is governed by the
same twisted group algebra in the same way. Therefore \eqref{eq:4.17} restricts to 
a spectrum preserving morphism of $\mc O (T_\fs)^{W_\fs \times X^L (\fs)}$-algebras
\[
(\cH^n / \cH^{n+1})^{X^L (\fs)} \to (J^n )^{X^L (\fs)}.
\]
Now a similar argument with Clifford theory for crossed product algebras shows that
\[
(\cH^n / \cH^{n+1})^{X^L (\fs)} \rtimes \mf R_\fs^\sharp \to 
(J^n )^{X^L (\fs)} \rtimes \mf R_\fs^\sharp
\]
is a spectrum preserving morphism of $\mc O (T_\fs)^{\Stab (\fs)}$-algebras.
By definition \cite[\S 5]{BaNi}, this means that the map
\begin{multline}\label{eq:5.2}
\phi'_{q_\fs} : \Big( \cH (T_\fs,W_\fs,q_\fs) \otimes \End_\C (V_\mu) 
\Big)^{X^L (\fs)} \rtimes \mf R_\fs^\sharp \to \\
\Big( J (X^* (T_\fs) \rtimes W_\fs) \otimes \End_\C (V_\mu) 
\Big)^{X^L (\fs)} \rtimes \mf R_\fs^\sharp 
\end{multline}
induced by $\phi_{q_\fs}$ is spectrum preserving with respect to filtrations. 

The same reasoning is valid with $\mc O (T_\fs) \rtimes W_\fs$ instead of
$\cH (T_\fs,W_\fs,q_\fs)$ -- it is simply the case $q_\fs = 1$ of the above.\\
(b) Recall that $T_\fs \cong X_\nr (L) / X_\nr (L,\omega)$. The torus
$T_\fs / X_\nr (L / L^\sharp Z(G))$ can be identified with 
\begin{equation}\label{eq:5.3}
X_\nr (L^\sharp Z(G)) / X_\nr (L,\omega) .
\end{equation}
Since the elements of $X_\nr (L,\omega)$ are trivial on $Z(L) \supset Z(G)$ and 
$L^\sharp \cap Z(G) \cong \mf o_F^\times$ is compact, \eqref{eq:5.3} factors as 
\[
X_\nr (L^\sharp) / X_\nr (L,\omega) \times X_\nr (Z(G)) = 
T_\fs^\sharp \times X_\nr (Z(G)) .
\]
By Theorem \ref{thm:1.1} the action of $X_\nr (L / L^\sharp Z(G)) \subset
X^L (\fs)$ on the algebras \eqref{eq:4.10} comes only from its action on the
torus $T_\fs$. Hence these three algebras do not change if we replace $T_\fs$
by \eqref{eq:5.3}. Equivalently, we may replace $T_\fs$ by 
$T_\fs^\sharp \times X_\nr (Z(G))$. It follows that 
\begin{multline*}
\Big( \cH (T_\fs,W_\fs,q_\fs) \otimes \End_\C (V_\mu) \Big)^{X^L (\fs)} 
\rtimes \mf R_\fs^\sharp \cong \\
\Big( \mc O (X_\nr (Z(G))) \otimes \cH (T_\fs^\sharp ,W_\fs,q_\fs) \otimes 
\End_\C (V_\mu) \Big)^{X^L (\fs)} \rtimes \mf R_\fs^\sharp .
\end{multline*}
The action of $\Stab (\fs)^+$ fixes $\mc O (X_\nr (Z(G)))$ pointwise,
so this equals 
\[
\Big( \cH (T_\fs^\sharp ,W_\fs,q_\fs) \otimes \End_\C (V_\mu) 
\Big)^{X^L (\fs)} \rtimes \mf R_\fs^\sharp \otimes \mc O (X_\nr (Z(G))) .
\]
The other two algebras in \eqref{eq:4.10} can be rewritten similarly.
By \eqref{eq:5.1} the morphisms $\phi_{q_\fs}$ and $\phi_1$ fix the respective
subalgebras $\mc O (X_\nr (Z(G)))$ pointwise. It follows that \eqref{eq:5.2}
decomposes as
\begin{multline*}
\phi_{q_\fs}^\sharp \otimes \mathrm{id} :
\Big( \cH (T_\fs^\sharp ,W_\fs,q_\fs) \otimes \End_\C (V_\mu) 
\Big)^{X^L (\fs)} \rtimes \mf R_\fs^\sharp \otimes \mc O (X_\nr (Z(G))) \to \\
\Big( J (X^* (T_\fs^\sharp) \rtimes W_\fs) \otimes \End_\C (V_\mu) 
\Big)^{X^L (\fs)} \rtimes \mf R_\fs^\sharp \otimes \mc O (X_\nr (Z(G))) ,
\end{multline*}
and similarly for $\phi'_1$. From part (a) we know that $\phi'_{q_\fs} = 
\phi_{q_\fs}^\sharp \otimes \mathrm{id}$ and $\phi'_1 = \phi_1^\sharp \otimes 
\mathrm{id}$ are spectrum preserving with respect to filtrations. So 
$\phi_{q_\fs}^\sharp$ and
\begin{multline*}
\phi_1^\sharp :
\Big( \mc O(T_\fs^\sharp) \rtimes W_\fs \otimes \End_\C (V_\mu) 
\Big)^{X^L (\fs)} \rtimes \mf R_\fs^\sharp \to \\
\Big( J (X^* (T_\fs^\sharp) \rtimes W_\fs) \otimes \End_\C (V_\mu) 
\Big)^{X^L (\fs)} \rtimes \mf R_\fs^\sharp 
\end{multline*}
have that property as well.
\end{proof}

With Theorem \ref{thm:4.4} we can show that the Hecke algebras for $G^\sharp$ and
for $G^\sharp Z(G)$ are stratified equivalent (see \cite{ABPS8}) to much simpler algebras.
Recall the subgroup $H_\lambda \subset L$ from \cite[Lemma 3.3]{ABPS4}.

\begin{thm}\label{thm:4.5}
\enuma{
\item The algebra $\cH (G^\sharp Z(G))^\fs$ is stratified equivalent with 
\begin{align*}
\bigoplus\nolimits_1^{[L : H_\lambda]} \Big( \mc O (T_\fs) \otimes 
\End_\C (V_\mu) \Big)^{X^L (\fs)} \rtimes W_\fs^\sharp .
\end{align*}
Here the action of $w \in W_\fs^\sharp$ is $\alpha_{(w,\gamma)}$ 
(as in Theorem \ref{thm:3.3}.a) for any $\gamma \in \Irr (L / L^\sharp Z(G))$
such that $(w,\gamma) \in \Stab (\fs)$.
\item The algebra $\cH (G^\sharp)^\fs$ is stratified equivalent with 
\[
\bigoplus\nolimits_1^{[L : H_\lambda]} \Big( \mc O (T_\fs^\sharp) 
\otimes \End_\C (V_\mu) \Big)^{X^L (\fs)} \rtimes W_\fs^\sharp ,
\] 
with respect to the same action of $W_\fs^\sharp$.
}
\end{thm}
\emph{Remark.} In principle one could factorize the above algebras according to
single Bernstein components for $G^\sharp Z(G)$ and $G^\sharp$. However, this would 
result in less clear formulas.
\begin{proof}
(a) Recall from Theorem \ref{thm:1.1} that
$\cH (G^\sharp Z(G))^\fs$ is Morita equivalent with
\begin{equation}\label{eq:3.5}
\bigoplus\nolimits_1^{[L : H_\lambda]} \Big( \cH (T_\fs,W_\fs,q_\fs) 
\otimes \End_\C (V_\mu) \Big)^{X^L (\fs)} \rtimes \mf R_\fs^\sharp .
\end{equation}
Consider the sequence of algebras 
\begin{equation}\label{eq:3.4}
\begin{aligned}
& \Big( \cH (T_\fs,W_\fs,q_\fs) \otimes \End_\C (V_\mu) 
\Big)^{X^L (\fs)} \rtimes \mf R_\fs^\sharp \\
& \to \Big( J (X^* (T_\fs) \rtimes W_\fs) \otimes \End_\C (V_\mu) 
\Big)^{X^L (\fs)} \rtimes \mf R_\fs^\sharp \\
& = \, \Big( J (X^* (T_\fs) \rtimes W_\fs) \otimes \End_\C (V_\mu) 
\Big)^{X^L (\fs)} \rtimes \mf R_\fs^\sharp \\
& \leftarrow \Big( \mc O (T_\fs) \rtimes W_\fs \otimes \End_\C (V_\mu) 
\Big)^{X^L (\fs)} \rtimes \mf R_\fs^\sharp .
\end{aligned}
\end{equation}
In Theorem \ref{thm:4.4}.a we proved that the map between the first two lines 
is spectrum preserving with respect 
to filtrations. The equality sign does nothing on the level of $\C$-algebras, 
but we use it to change the $\mc O (T_\fs)^{\Stab (\fs)}$-module structure, 
such that the map from
\begin{equation}\label{eq:4.16}
\Big( \mc O (T_\fs) \rtimes W_\fs \otimes \End_\C (V_{\mu}) 
\Big)^{X^L (\fs)} \rtimes \mf R_\fs^\sharp 
\end{equation}
becomes $\mc O (T_\fs)^{\Stab (\fs)}$-linear. By Theorem \ref{thm:4.4}.a 
that map is also spectrum preserving with respect to filtrations. 

Every single step in the above sequence is an instance of stratified equivalence
from \cite{ABPS8} (the second step by definition), so $\cH (G^\sharp Z(G))^\fs$ 
is stratified equivalent with a direct sum of $[L : H_\lambda ]$ copies of 
\eqref{eq:4.16}. Since $\chi_\gamma \in T_\fs$ in 
\eqref{eq:4.14} is $W_\fs$-invariant, the actions of $X^L (\fs)$ and 
$W_\fs$ on $\mc O (T_\fs) \otimes \End_\C (V_\mu)$ commute. This observation 
and \eqref{eq:1.6} allow us to identify \eqref{eq:4.16} with
\begin{equation}\label{eq:3.6}
\Big( \Big( \mc O (T_\fs) \otimes \End_\C (V_\mu) 
\Big)^{X^L (\fs)} \rtimes W_\fs \Big) \rtimes \mf R_\fs^\sharp =
\Big( \mc O (T_\fs) \otimes \End_\C (V_\mu) 
\Big)^{X^L (\fs)} \rtimes W_\fs^\sharp . 
\end{equation}
The description of the action of $W_\fs^\sharp$ can be derived from 
Theorem \ref{thm:1.3}. \\
(b) This follows from Theorem \ref{thm:1.2} and the same proof as for part (a).
\end{proof}

\section{Extended quotients for inner forms of $\GL_n$}

It turns out that via \eqref{eq:1.30} any Bernstein component for $G$ can be described 
in a canonical way with an extended quotient. Before we prove that, we recall the 
parametrization of irreducible representations of $\cH (T_\fs,W_\fs,q_\fs)$. 

Let $\check G_\fs$ be the complex reductive group with root datum $(X^* (T_\fs ), 
R_\fs ,X_* (T_\fs ),R_\fs^\vee )$, it is isomorphic to $\prod_i \GL_{e_i}(\C)$,
embedded in $\check G = \GL_{md}(\C)$ as
\[
\check G_\fs = Z_{\check G}(\check L) =
Z_{\GL_{md}(\C)} \bigg(\prod\nolimits_i\GL_{m_id}(\C)^{e_i}\bigg). 
\]
Recall that a Kazhdan--Lusztig triple for $\check G_\fs$ consists of:
\begin{itemize}
\item a unipotent element $u = \prod_i u_i \in \check G_\fs$;
\item a semisimple element $t_q \in \check G_\fs$ with $t_q u t_q^{-1} = u^{q_\fs} := 
\prod_i u_i^{q_i}$;
\item a representation $\rho_q \in \Irr \big( \pi_0 (Z_{\check G_\fs}(t_q,u)) \big)$
which appears in the homology of variety of Borel subgroups of $\check G_\fs$ 
containing $\{t_q,u\}$.
\end{itemize}
Typically such a triple is considered up to $\check G_\fs$-conjugation, 
we denote its equivalence class by $[t_q,u,\rho_q]_{\check G_\fs}$. These equivalence 
classes parametrize $\Irr (\cH (T_\fs,W_\fs,q_\fs))$ in a natural way, see \cite{KaLu}. 
We denote that by
\begin{equation}\label{eq:5.17}
[t_q,u,\rho_q]_{\check G_\fs} \mapsto \pi (t_q,u,\rho_q) .
\end{equation}
Recall from \cite[\S 7]{ABPS6} that an affine Springer parameter for $\check G_\fs$ 
consists of:
\begin{itemize}
\item a unipotent element $u = \prod_i u_i \in \check G_\fs$;
\item a semisimple element $t \in Z_{\check G_\fs}(u)$;
\item a representation $\rho \in \Irr \big( \pi_0 (Z_{\check G_\fs}(t,u)) \big)$ 
which appears in the homology of variety of Borel subgroups of $\check G_\fs$ 
containing $\{t,u\}$.
\end{itemize}
Again such a triple is considered up to $\check G_\fs$-conjugacy, and then denoted
$[t,u,\rho ]_{\check G_\fs}$. Kato \cite{Kat} established a natural bijection 
between such equivalence classes and $\Irr (\mc O (T_\fs) \rtimes W_\fs)$, say 
\begin{equation}\label{eq:5.18}
[t,u,\rho]_{\check G_\fs} \mapsto \tau (t,u,\rho) .
\end{equation}
For a more explicit description, we note that $Z_{\check G_\fs}(t)$ is a connected
reducitive group with Weyl group $W_{\fs,t}$, and  that $(u,\rho)$ represents a
Springer parameter for $W_{\fs,t}$. Via the classical Springer correspondence
\begin{equation}\label{eq:5.4}
(u,\rho) \text{ determines an irreducible } W_{\fs,t}\text{-representation } \pi (u,\rho). 
\end{equation}
Then \cite{Kat} and \eqref{eq:5.18} work out to
\begin{equation}\label{eq:5.5}
\tau (t,u,\rho) = \mr{ind}_{\mc O (T_\fs) \rtimes W_{\fs,t}}^{\mc O (T_\fs) \rtimes W_\fs}
(\C_t \otimes \pi (u,\rho)) .
\end{equation}
From \cite[\S 2.4]{KaLu} we get a canonical bijection between Kazhdan--Lusztig 
triples and affine Springer parameters:
\begin{equation}\label{eq:5.19}
[t_q,u,\rho_q]_{\check G_\fs} \longleftrightarrow [t,u,\rho]_{\check G_\fs} .
\end{equation}
Basically it adjusts $t_q$ in a minimal way so that it commutes with $u$, 
and then there is only one consistent way to modify $\rho_q$ to $\rho$.

Via Lemma \ref{lem:3.4} the algebra homomorphisms \eqref{eq:3.3} 
give rise to a bijection
\begin{equation}\label{eq:5.20}
\Irr \big( \cH (T_\fs,W_\fs,q_\fs) \big) \longleftrightarrow 
\Irr \big( \mc O (T_\fs) \rtimes W_\fs \big) .
\end{equation}
We showed in \cite[(90)]{ABPS6} that \eqref{eq:5.20} is none other than the composition
of \eqref{eq:5.19} with \eqref{eq:5.18} and the inverse of \eqref{eq:5.17}:
\begin{equation}\label{eq:5.21}
\pi (t_q,u,\rho_q) \longleftrightarrow \tau (t,u,\rho) .
\end{equation}

\begin{thm}\label{thm:4.6}
The Morita equivalence $\cH (G)^\fs \sim_M \cH (T_\fs,W_\fs,q_\fs)$ and 
\eqref{eq:5.20} give rise to a bijection
\begin{equation}\label{eq:4.37}
\Irr^\fs (G) \longleftrightarrow T_\fs \q W_\fs
\end{equation}
with the following properties:
\begin{enumerate}
\item Let $T_{\fs,\uni}$ be the maximal compact subtorus of $T_\fs$ and let 
$\Irr_\temp (G) \subset \Irr (G)$ be the subset of tempered representations. Then
\eqref{eq:4.37} restricts to a bijection $\Irr_\temp^\fs (G) 
\longleftrightarrow T_{\fs,\uni} \q W_\fs$. 
\item \eqref{eq:4.37} can be obtained from its restriction to tempered 
representations by analytic continuation, as in \cite{ABPS1}. For instance, suppose
that $\sigma \in \Irr_\temp^{\fs_{L'}}(L')$ for some standard parabolic $P' = L' U'
\supset P = L U$, and that $I_{P'}^G (\sigma \otimes \chi)$ is mapped to 
$\tau (t \chi,u,\rho)$ for almost all unitary $\chi \in X_\nr (L')$. Then, whenever
$\chi_\nr (L')$ and $I_{P'}^G (\sigma \otimes \chi)$ is irreducible, it is mapped to
$\tau (t \chi,u,\rho)$.
\item If $\pi \in \Irr^\fs_\temp (G)$ is mapped to $[t,\rho'] \in T_{\fs,\uni} \q 
W_\fs $ and has cuspidal support $W_\fs \sigma \in T_\fs / W_\fs$, then $W_\fs t$
is the unitary part of $W_\fs \sigma$, with respect to the polar decomposition
\[
T_\fs = T_{\fs,\uni} \times \Hom_\Z (X^* (T_\fs),\R_{>0}) . 
\]
\item In the notation of (3), suppose that the parameter of $\rho' \in \Irr (W_{\fs,t})$ 
in the classical Springer correspondence \eqref{eq:5.4} involves a unipotent class $[u]$ 
which is distinguished in a Levi subgroup $\check M \subset Z_{\check G_\fs}(t)$. 
Then $\pi = I_{PM}^G (\delta)$, where $M \supset L$ is the unique standard Levi subgroup 
of $G$ corresponding to $\check M$ and $\delta \in \Irr^{[L,\omega]_M}_\temp (M)$ 
is square-integrable modulo centre.
\end{enumerate}
Moreover \eqref{eq:4.37} is the unique bijection with the properties (1)--(4).
\end{thm}
\begin{proof}
The Morita equivalence \eqref{eq:1.30} gives a bijection
\begin{equation}\label{eq:4.38}
\Irr^\fs (G) \longleftrightarrow \Irr (\cH (T_\fs,W_\fs,q_\fs)) .
\end{equation}
Via Lemmas \ref{lem:3.4} and \ref{lem:B.6} the right hand side is in bijection
with 
\begin{equation}\label{eq:5.6}
\Irr (\cO (T_\fs) \rtimes W_\fs) \cong T_\fs \q W_\fs  . 
\end{equation}
In this way we define the map \eqref{eq:4.37}.\\
(1) It is easy to check from \cite[Th\'eor\`eme 4.6]{Sec3} and \cite[Theorem C]{SeSt6} 
that the Morita equivalence $\cH (G)^\fs \sim_M \cH (T_\fs,W_\fs,q_\fs)$ preserves
the canonical involution and trace (maybe up to a positive scalar). Accepting that,
\cite[Theorem 10.1]{DeOp} says that the ensueing bijection between $\Irr^\fs (G)$ and 
$\Irr \big( \cH (T_\fs,W_\fs,q_\fs) \big)$ respects the subsets of tempered representations.
By \cite[Proposition 9.3]{ABPS6} the latter subset corresponds to the set of 
Kazhdan--Lusztig triples such that the $t$ in \eqref{eq:5.19} lies in $T_{\fs,\uni}$.\\
(2) Consider the bijection \eqref{eq:5.20} and its formulation \eqref{eq:5.21}. 
Here the representations are tempered if and only if $t \in T_\fs$ is unitary.
Thus \eqref{eq:5.21} for tempered representations determines the bijection
\eqref{eq:5.20}, by analytic continuation (in the parameters $t$ and $t_q$)
of the formula. 

The relation between $\Irr^\fs (G)$ and $\Irr_\temp^\fs (G)$ is similar, see
\cite[Proposition 2.1]{ABPS2}. Hence \eqref{eq:4.38} is can also be deduced
from its restriction to tempered representations, with the method from
\cite[\S 4]{ABPS2}.\\
(3) In \cite[Th\'eor\`eme 4.6]{Sec3} a $\fs_L$-type $(K_L,\lambda_L)$ is constructed,
with
\[
e_{\lambda_L} \cH (L) e_{\lambda_L} \cong \cO (T_\fs) \otimes \End_\C (V_\lambda) . 
\]
It \cite{SeSt6} it is shown that it admits a cover $(K_G,\lambda_G)$ with 
\[
e_{\lambda_G} \cH (G) e_{\lambda_G} \cong 
\cH (T_\fs,W_\fs,q_\fs) \otimes \End_\C (V_\lambda) ,
\]
see also \cite[\S 4.1]{ABPS4}. It follows from \cite[Proposition 3.15]{ABPS4} that
\eqref{eq:4.38} arises from this cover of a $\fs_L$-type.
With \cite[\S 7]{BuKu3} this implies that \eqref{eq:4.37} translates the 
cuspidal support of a $(\pi,V_\pi) \in \Irr^\fs (G)$ to the unique $W_\fs t_q \in 
T_\fs / W_\fs$ such that $e_{\lambda_G} V_\pi$ is a subquotient of 
$\ind_{\cO (T_\fs)}^{\cH (T_\fs,W_\fs,q_\fs)} (\C_{t_q}) \otimes V_\lambda$.
It follows from \cite[(33) and Lemma 7.1]{ABPS6} that the bijection \eqref{eq:5.21}
sends any tempered irreducible subquotient of  
$\ind_{\cO (T_\fs)}^{\cH (T_\fs,W_\fs,q_\fs)} (\C_{t_q})$ to an irreducible
$\cO (T_\fs) \rtimes W_\fs$-representation with $\cO (T_\fs)$-weights $W_\fs (t_q
\, |t_q|^{-1})$. The associated element of $T_\fs \q W_\fs$ is then
$[t = t_q \, |t_q|^{-1},\rho]$ with $\rho \in \Irr (W_{\fs,t})$.\\
(4) Since $W_{\fs,t}$ is a direct product of symmetric groups, the representations $\rho'$ 
of the component groups in the Springer parameters are all trivial. By \eqref{eq:5.4}, 
\eqref{eq:5.5} and \eqref{eq:5.21} the $\cH (T_\fs,W_\fs,q_\fs)$-representation 
associated to $[t,\rho']$ is $\pi (t_q,u,\mathrm{triv})$. Then $(t_q,u,\mathrm{triv})$ is 
also a Kazhdan--Lusztig triple for $\cH (T_\fs,W_{\fs,M},q_\fs)$ and by \cite[\S 7.8]{KaLu}
\[
\pi (t_q,u,\mathrm{triv}) = \ind_{\cH^M}^\cH \pi_M (t_q,u,\mathrm{triv}). 
\]
By \cite[Proposition 9.3]{ABPS6} (see also \cite[Theorem 8.3]{KaLu}) 
$\pi_M (t_q,u,\mathrm{triv})$ is essentially square-integrable and tempered, that is,
square-integrable modulo centre. 

Since $(K_G,\lambda_G)$ is a a cover of a $\fs_L$-type $(K_L,\lambda_L)$, there is
a a $[L,\omega]_M$-type $(K_M,\lambda_M)$ which covers $(K_L,\lambda_L)$ and is
covered by $(K_G,\lambda_G)$. By \cite[Proposition 16.6]{ABPS6} 
$\pi_M (t_q,u,\mathrm{triv})$ corresponds to a $M$-representation $\delta$ which
is square-integrable modulo centre. By \cite[Corollary 8.4]{BuKu3} the bijection 
\eqref{eq:4.38} respects parabolic induction, so $\pi (t_q,u,\mathrm{triv})$
corresponds to $I_{PM}^G (\delta)$.\\
Now we check that \eqref{eq:4.37} is canonical in the specified sense. By (1) and
(2) it suffices to do so for tempered representations. For $\pi \in 
\Irr^\fs_\temp (G)$, property (3) determines the $W_\fs$-orbit $W_\fs t$. Fix a
$t$ in this orbit. By a result of Harish-Chandra \cite[Proposition III.4.1]{Wal}
there are a Levi subgroup $M \subset G$ containing $L$ and a square-integrable
(modulo centre) representation $\delta \in \Irr (M)$ such that $\pi$ is a 
subquotient of $I_{PM}^G (\delta)$. Moreover $(M,\delta)$ is unique up to conjugation.

For $t \in T_{\fs,\uni} ,\; W_{\fs,t}$ is a product of symmetric groups $S_e$ and
$Z_{\check G_\fs}(t)$ is a product of groups of the form $\GL_e (\C)$. Hence the 
Springer correspondence for $W_{\fs,t}$ is a bijection between $\Irr (W_{\fs,t})$ and 
unipotent classes in $Z_{\check G_\fs}(t)$. Every Levi subgroup of $\GL_e (\C)$ 
has a unique distinguished unipotent class, and these exhaust the unipotent classes
in $\GL_e (\C)$. Hence $\Irr (W_{\fs,t})$ is also in canonical bijection with the set 
of conjugacy classes of Levi subgroups $\check M \subset Z_{\check G_\fs}(t)$. 

Viewed in this light, properties (3) and (4) entail that for every pair $(\check M,t)$
as above there is precisely one square-integrable modulo centre $\delta \in \Irr (M)$
such that $W_\fs t$ is the unitary part of the cuspidal support of 
$I_{PM}^G (\delta)$. Thus (3) and (4) determine the (tempered) $G$-representation 
associated to $[t,\rho] \in T_{\fs,\uni} \q W_\fs$.
\end{proof}

\section{Twisted extended quotients for inner forms of $SL_n$}
\label{sec:extquot}

Twisted extended quotients appear naturally in the description of the 
Bernstein components for $L^\sharp Z(G)$ and $L^\sharp$.

\begin{lem}\label{lem:4.7}
Let $\fs_L = [L,\omega]_L$ and define a two-cocycle $\kappa_\omega$ by \eqref{eq:2.21}.
\enuma{
\item Equation \eqref{eq:2.2} for $L$ determines bijections
\begin{align*}
& (T_\fs \q X^L (\fs) )_{\kappa_\omega} \to \Irr^{\fs_L} (L^\sharp Z(G)) , \\
& (T_\fs \q X^L (\fs) X_\nr (L / L^\sharp) )_{\kappa_\omega} = (T_\fs \q X^L (\fs) 
)_{\kappa_\omega} / X_\nr (L^\sharp Z(G) / L^\sharp) \to \Irr^{\fs_L} (L^\sharp) .
\end{align*}
\item The induced maps
\[
\Irr^{\fs_L} (L^\sharp Z(G)) \to T_\fs / X^L (\fs) \quad \text{and} \quad
\Irr^{\fs_L} (L^\sharp) \to T_\fs / X^L (\fs) X_\nr (L / L^\sharp)
\]
are independent of the choice of $\kappa_\omega$.
\item Let $T_{\fs,\uni}$ be the real subtorus of unitary representations in
$T_\fs$.
The subspace of tempered representations $\Irr_\temp^{\fs_L}(L^\sharp Z(G))$
corresponds to $(T_{\fs,\uni} \q X^L (\fs) )_{\kappa_\omega}$. Similarly
$\Irr^{\fs_L}_\temp (L^\sharp)$ is obtained by restricting the second line
of part (a) to $T_{\fs,\uni}$.
}
\end{lem}
\begin{proof}
(a) Apart from the equality, this is a reformulation of the last page of Section 
\ref{sec:tori}. For the equality, we note that by \eqref{eq:4.4} the action of
\[
X_\nr (L^\sharp Z(G) / L^\sharp) \cong X_\nr (L / L^\sharp) / (X^L (\fs) \cap
X_\nr (L / L^\sharp)) 
\]
on $(T_\fs \q X^L (\fs) )_{\kappa_\omega}$ is free. Hence the isotropy groups for
the action of $X^L (\fs) X_\nr (L / L^\sharp)$ are the same as for $X^L (\fs)$,
and we can use the same 2-cocycle $\kappa_\omega$ to construct a twisted extended
quotient. \\
(b) By \eqref{eq:2.2} a different choice of $\kappa_\omega$ in part (a) would only 
lead to the choice of another irreducible summand of $\Res^G_{G^\sharp}(\pi)$ for 
$\pi \in T_\fs$, and similarly for $G^\sharp Z(G)$.\\
(c) Since $\omega$ is supercuspidal, the set of tempered representations in
$T_\fs = \Irr^{\fs_L}(L)$ is $T_{\fs,\uni}$.
In the decomposition \eqref{eq:2.2}, an irreducible representation of $L^\sharp$
or $L^\sharp Z(G)$ is tempered if and only if it is contained in a tempered
$L$-representation $\omega \otimes \chi$. This proves the statement for $L$. 
The claim for $L^\sharp$ follows
upon dividing out the free action of $X_\nr (L^\sharp Z(G) / L^\sharp)$.
\end{proof}

The subgroup $X^L (\omega,V_\mu)$ acts trivially on
$\mc O (T_\fs) \otimes \End_\C (V_\mu)$, and for that reason it can be pulled
out of the extended quotient from Lemma \ref{lem:4.7}.

\begin{lem}\label{lem:4.8}
There are bijections
\begin{align*}
& (T_\fs \q X^L (\fs) )_{\kappa_\omega} \longleftrightarrow \big( T_\fs \q X^L (\fs) 
/ X^L (\omega,V_\mu) \big)_{\kappa_\omega} \times \Irr (X^L (\omega,V_\mu)) , \\
& (T_\fs \q X^L (\fs) X_\nr (L / L^\sharp) )_{\kappa_\omega} \longleftrightarrow 
\big( T_\fs \q X \big)_{\kappa_\omega} \times \Irr (X^L (\omega,V_\mu)),
\end{align*}
where $X = X^L (\fs) X_\nr (L / L^\sharp) / X^L (\omega,V_\mu)$.
They fix the coordinates in $T_\fs$.
\end{lem}
\begin{proof}
In Lemma \ref{lem:4.7}.a we saw that $(T_\fs \q X^L (\fs) )_{\kappa_\omega}$ is in 
bijection with \\ $\Irr \big( \cH (L^\sharp Z(G))^{\fs_L} \big)$. 
By \cite[(169)]{ABPS4} $\cH (L^\sharp Z(G))^{\fs_L}$ is Morita equivalent with 
$(\cH (L)^{\fs_L} )^{X^L (\fs)}$ and with the subalgebra 
\begin{multline}\label{eq:4.22}
e^\fs_L \cH (L)^{X^L (\fs)} e^\fs_L \cong \big( \mc O (T_\fs) \otimes 
\End_\C (e^\fs_L V_\omega) \big)^{X^L (\fs)} \\
\cong \bigoplus\nolimits_{a \in [L / H_\lambda]} \big( \mc O (T_\fs) \otimes 
\End_\C (V_\mu) )^{X^L (\fs) / X^L (\omega,V_\mu)} .
\end{multline}
Here the $X^L (\fs)$-action on the middle term comes from an isomorphism 
\[
\End_\C (e^\fs_L V_\omega) \cong \End_\C (V_\mu) \otimes \C [L / H_\lambda] 
\otimes \C[L / H_\lambda]^* .
\]
We recall that by \cite[Lemma 3.5]{ABPS4} there is a group isomorphism
\begin{equation}\label{eq:4.21} 
L / H_\lambda \cong \Irr (X^L (\omega,V_\mu)) . 
\end{equation}
By the above Morita equivalences, Lemma \ref{lem:4.7}.a and Clifford theory, the set 
\begin{equation}\label{eq:4.45}
\{ \rho \in \Irr (\C [X^L (\omega),\kappa_\omega]) : \rho \big|_{X^L (\omega,V_\mu)}
= \mathrm{triv} \} 
\end{equation}
parametrizes the irreducible representations of \eqref{eq:4.22} associated to
a fixed $\chi \in T_\fs$ and the trivial character of $X^L (\omega,V_\mu)$.

For any $\chi \in T_\fs$, the $X^L (\fs) / X^L (\omega, V_\mu)$-stabilizer of the
irreducible representation $\C_\chi \otimes V_\mu$ of $\mc O (T_\fs) \otimes \End (V_\mu)$
is $X^L (\omega) / X^L (\omega,V_\mu)$. It follows that the irreducible representations
of \eqref{eq:4.22} with fixed $a \in [L / H_\lambda]$ and fixed $\mc O (T_\fs)$-character
$\chi$ are in bijection with $\Irr \big( \End_\C (V_\mu)^{X^L (\omega) / X^L (\omega,V_\mu)}
\big)$. Comparing with \eqref{eq:4.45}, we see that every irreducible representation of
$\C [X^L (\omega)/X^L (\omega,V_\mu),\kappa_\omega]$ appears in $V_\mu$. This is 
equivalent to each irreducible representation of $\End_\C (V_\mu) \rtimes 
X^L (\omega)/X^L (\omega,V_\mu)$ having nonzero vectors fixed by 
$X^L (\omega)/X^L (\omega,V_\mu)$. Thus Lemma \ref{lem:B.7} can be applied to
$X^L (\omega)/X^L (\omega,V_\mu)$ acting on $\mc O (T_\fs) \otimes \End_\C (V_\mu)$,
and it shows that the irreducible representations on the right hand side of 
\eqref{eq:4.22} are in bijection with 
\[
(T_\fs \q X^L (\fs) /
X^L (\omega,V_\mu) )_{\kappa_\omega} \times \Irr (X^L (\omega,V_\mu)).
\]
The second bijection follows by dividing out the free action of $X_\nr (L^\sharp
Z(G) / L^\sharp)$, as in the proof of Lemma \ref{lem:4.7}.a.
\end{proof}

As a result of the work in Section \ref{sec:geomEquiv}, twisted extended quotients
can also be used to describe the spaces of irreducible representations of
$G^\sharp Z(G)$ and $G^\sharp$. Let us extend $\kappa_\omega$ to a two-cocycle of 
$\Stab (\fs)$, trivial on the normal subgroup $W_\fs \times X^L (\omega,V_\mu)$, by
\begin{equation}\label{eq:4.30}
J(\gamma,\omega) J(\gamma',\omega) = \kappa_\omega (\gamma,\gamma') 
J(\gamma \gamma', \omega) \qquad \gamma, \gamma' \in X^G (\fs) . 
\end{equation}
\begin{thm}\label{thm:4.9}
\enuma{
\item Lemmas \ref{lem:B.6} and \ref{lem:B.7} gives rise to bijections
\begin{align*}
& \big( T_\fs \q \Stab (\fs) / X^L (\omega,V_\mu) \big)_{\kappa_\omega} 
\to \Irr \Big( \big( \mc O (T_\fs) \otimes \End_\C (V_\mu) \big)^{X^L (\fs)} 
\rtimes W_\fs^\sharp \Big) ,\\
& \big( T_\fs \q \Stab (\fs) X_\nr (L / L^\sharp) / X^L (\omega,V_\mu) 
\big)_{\kappa_\omega} \to \Irr \Big( \big( \mc O (T^\sharp_\fs) \otimes 
\End_\C (V_\mu) \big)^{X^L (\fs)} \rtimes W_\fs^\sharp \Big) .
\end{align*}
\item The stratified equivalences from \ref{thm:4.5} provide bijections
\begin{align*}
&  ( T_\fs \q \Stab (\fs) )_{\kappa_\omega} \to \big( T_\fs \q \Stab (\fs) / 
X^L (\omega,V_\mu) \big)_{\kappa_\omega} \times \Irr (X^L (\omega,V_\mu)) \to
\Irr^\fs (G^\sharp Z(G)) , \\
& ( T_\fs \q \Stab (\fs) X_\nr (L / L^\sharp) )_{\kappa_\omega} \to \big( T_\fs \q S 
\big)_{\kappa_\omega} \times \Irr (X^L (\omega,V_\mu)) \to \Irr^\fs (G^\sharp) ,
\end{align*}
where $S = \Stab (\fs) X_\nr (L / L^\sharp) / X^L (\omega,V_\mu)$. 
\item In part (b) $\Irr^\fs_\temp (G^\sharp Z(G))$ (respectively 
$\Irr^\fs_\temp (G^\sharp)$) corresponds to the same extended quotient,
only with $T_{\fs,\uni}$ instead of $T_\fs$.
} 
\end{thm}
\begin{proof}
In each of the three parts the second claim follows from the first upon
dividing out the action of $X_\nr (L^\sharp Z(G) / L^\sharp)$, like in
Lemma \ref{lem:4.7}.a\\
(a) In the proof of Lemma \ref{lem:4.8} we exhibited a bijection
\[
\big( T_\fs \q \Stab (\fs) / X^L (\omega,V_\mu) \big)_{\kappa_\omega}
\longleftrightarrow \Irr \Big( \big( \mc O (T_\fs) \otimes \End_\C (V_\mu) 
\big)^{X^L (\fs)} \Big) .
\]
With Lemma \ref{lem:B.7} we deduce a Morita equivalence 
\begin{equation}\label{eq:4.24}
\big( \mc O (T_\fs) \otimes \End_\C (V_\mu) \big)^{X^L (\fs)} \sim_M
\big( \mc O (T_\fs) \otimes \End_\C (V_\mu) \big) 
\rtimes (X^L (\fs) / X^L (\omega,V_\mu)) .
\end{equation}
In the notation of \eqref{eq:4.23} this means that
$p := p_{X^L (\fs) / X^L (\omega,V_\mu)}$ is a full idempotent in the right hand 
side of \eqref{eq:4.24}, that is, the two-sided ideal it generates is the 
entire algebra. Then $p$ is also full in
\begin{equation}\label{eq:4.25}
\big( \mc O (T_\fs) \otimes \End_\C (V_\mu) \big) 
\rtimes (\Stab (\fs) / X^L (\omega,V_\mu)) , 
\end{equation}
which implies that \eqref{eq:4.25} is Morita equivalent with
\begin{multline*}
p \Big( \big( \mc O (T_\fs) \otimes \End_\C (V_\mu) \big) 
\rtimes (\Stab (\fs) / X^L (\omega,V_\mu)) \Big) p \cong \\
\big( \mc O (T_\fs) \otimes \End_\C (V_\mu) \big)^{X^L (\fs) / X^L (\omega,V_\mu)}
\rtimes (\Stab (\fs) / X^L (\fs)) .
\end{multline*}
As a direct consequence of \eqref{eq:1.6}, \eqref{eq:1.11} and \eqref{eq:1.12},
\[
\Stab (\fs) / X^L (\fs) \cong W_\fs^\sharp . 
\]
In this way we reach the algebra featuring in part (a). By the above Morita
equivalence, its irreducible representations are in bijection with those of
\eqref{eq:4.25}. Apply Lemma \ref{lem:B.6}.a to the latter algebra.\\
(b) All the morphisms in \eqref{eq:3.4} are spectrum preserving with respect
to filtrations. In combination with the other remarks in the proof of
Theorem \ref{thm:4.5}.a this gives a bijection
\begin{equation}\label{eq:4.26}
\Irr^\fs (G^\sharp Z(G)) \to \Irr \Big( \big( \mc O (T_\fs) \otimes \End_\C (V_\mu) 
\big)^{X^L (\fs)} \rtimes W_\fs^\sharp \Big) \times [L / H_\lambda] .
\end{equation}
By part (a) and \eqref{eq:4.21} the right hand side of \eqref{eq:4.26} 
is in bijection with
\begin{equation}\label{eq:4.36}
\big( T_\fs \q \Stab (\fs) / X^L (\omega,V_\mu) \big)_{\kappa_\omega} 
\times \Irr (X^L (\omega,V_\mu)) .
\end{equation}
The group $X^L (\fs)$ acts on $\C [L]$ by pointwise multiplication of functions
on $L$. That gives rise to actions on $\C [L / H_\lambda]$ and on
\[
\End_\C (\C [L / H_\lambda]) \cong \C [L / H_\lambda] \otimes \C [L / H_\lambda]^* .
\]
Regarding $\C [L / H_\lambda]$ as the algebra $\bigoplus_{a \in [H / H_\lambda]} \C$, 
\eqref{eq:4.21} leads to an isomorphism
\begin{equation}\label{eq:4.28}
\begin{split}
\big( \mc O (T_\fs) \otimes \End_\C (V_\mu) 
\big)^{X^L (\fs)} \rtimes W_\fs^\sharp \otimes \C [L / H_\lambda] \cong \\ 
\big( \mc O (T_\fs) \otimes \End_\C (V_\mu \otimes \C[L / H_\lambda]) 
\big)^{X^L (\fs)} \rtimes W_\fs^\sharp .
\end{split}
\end{equation}
We note that \eqref{eq:4.36} is also the space of irreducible representations of 
\eqref{eq:4.28}. In the proof of Lemma \ref{lem:4.8} we encountered a bijection
\[
(T_\fs \q X^L (\fs) )_{\kappa_\omega} \longleftrightarrow 
\Irr \Big( \big( \mc O (T_\fs) \otimes \End_\C (V_\mu \otimes \C[L / H_\lambda]) 
\big)^{X^L (\fs)} \Big) .
\]
It implies a Morita equivalence 
\[
\big( \mc O (T_\fs) \otimes \End_\C (V_\mu \otimes \C[L / H_\lambda]) 
\big)^{X^L (\fs)} \sim_M \big( \mc O (T_\fs) \otimes 
\End_\C (V_\mu \otimes \C[L / H_\lambda]) \big) \rtimes X^L (\fs) .
\]
Just as in the proof of part (a), this extends to 
\begin{equation}\label{eq:4.27}
\begin{split}
\big( \mc O (T_\fs) \otimes \End_\C (V_\mu \otimes \C[L / H_\lambda]) 
\big)^{X^L (\fs)} \rtimes W_\fs^\sharp \sim_M \\
\big( \mc O (T_\fs) \otimes 
\End_\C (V_\mu \otimes \C[L / H_\lambda]) \big) \rtimes \Stab (\fs) .
\end{split}
\end{equation}
Finally we apply Lemma \ref{lem:B.6}.a to the right hand side and we 
combine it with \eqref{eq:4.27}, \eqref{eq:4.28} and \eqref{eq:4.26}.\\
(c) The first bijection in part (b) obviously preserves the subspaces associated
to $T_{\fs,\uni}$. We need to show that the second bijection sends them
to $\Irr^\fs_\temp (G^\sharp Z(G))$. This is a property of the
geometric equivalences in Theorem \ref{thm:4.5}, as we will now check.

We may and will assume that $\omega$ is unitary, or equivalently that
it is tempered. The Morita equivalence between 
$\cH (G^\sharp Z(G))^\fs$ and \eqref{eq:3.5} is induced by an idempotent
$e^\sharp_{\lambda_{G^\sharp Z(G)}} \in \cH (G^\sharp Z(G))$, see Theorem 
\ref{thm:1.1}. Its construction (which starts around \eqref{eq:1.7})
shows that eventually it comes from a central idempotent in the algebra
of a profinite group, so it is a self-adjoint element. Hence, by 
\cite[Theorem A]{BHK} this Morita equivalence preserves temperedness. The notion 
of temperedness in \cite{BHK} agrees with temperedness for representations of 
affine Hecke algebras (see \cite{Opd}) because both are based on
the Hilbert algebra structure and the canonical tracial states on these algebras. 

The sequence of algebras \eqref{eq:3.4} is derived from its counterpart for\\
$\cH (T_\fs,W_\fs,q_\fs) \otimes \End_\C (V_\mu)$. By Theorem \ref{thm:4.6} 
that one matches tempered representations with $T_{\fs,\uni} \q W_\fs$.
By Clifford theory any irreducible representation $\pi$ of 
\begin{equation}\label{eq:4.29}
\big( \cH (T_\fs,W_\fs,q_\fs) \otimes \End_\C (V_\mu) \big)^{X^L (\fs)}
\rtimes \mf R_\fs^\sharp
\end{equation}
is contained in a sum of irreducible representations $\tilde \pi$ of 
$\cH (T_\fs,W_\fs,q_\fs) \otimes \End_\C (V_\mu)$, which are all in the same
$\Stab (\fs)$-orbit. Temperedness of $\pi$ depends only on the action of 
the subalgebra $\mc O(T_\fs) \cong \C [X^* (T_\fs)]$, and in fact can already
be detected on $\C[X]$ for any finite index sublattice $X \subset X^* (T_\fs)$.
The analogous statement for \eqref{eq:4.29} holds as well, with
$X = X^* (T_\fs / X^L (\fs))$, and it is stable under the action of $\Stab (\fs)$. 
Consequently $\pi$ is tempered if and only if $\tilde \pi$ is tempered. 

These observations imply that the sequence of algebra homomorphisms 
\eqref{eq:3.4} preserves temperedness of irreducible representations, and
that it maps such representations of \eqref{eq:4.29} to irreducible 
representations of \eqref{eq:3.6} with $\mc O (T_\fs / X^L (\fs))$-weights
in $T_{\fs,\uni} / X^L (\fs)$. 

Now we invoke this property for every $a \in L / H_\lambda \cong
\Irr (X^L (\omega,V_\mu))$ and we deduce that the second map in part (b)
has the required property with respect to temperedness.
\end{proof}

We will work out what Theorem \ref{thm:4.9} says for a single Bernstein component of 
$G^\sharp$. To this end, we first analyse what parabolic induction from $L^\sharp$ 
to $G^\sharp$ looks like in the setting of Theorem \ref{thm:1.2}.

\begin{thm}\label{thm:1.4}
\enuma{
\item There exist idempotents $e^\fs_L \in \cH (L), e^\fs_{L^\sharp} \in \cH (L^\sharp)$
such that $\cH (L^\sharp )^\fs$ is Morita equivalent with
\begin{multline*}
e^\fs_{L^\sharp} \cH (L^\sharp) e^\fs_{L^\sharp} \cong e^\fs_L \cH (L)^{X^L (\fs) 
X_\nr (L / L^\sharp)} e^\fs_L \cong
\Big( \mc O (T_\fs^\sharp) \otimes \End_\C (V_\mu \otimes \C [L / H_\lambda]) 
\Big)^{X^L (\fs)} \\
\cong \bigoplus\nolimits_{a \in [L / H_\lambda]} \Big( \mc O (T_\fs^\sharp) \otimes 
\End_\C (V_\mu) \Big)^{X^L (\fs) / X^L (\omega,V_\mu)} . 
\end{multline*}
\item Under the equivalences from part (a) and Theorem \ref{thm:1.2}, the normalized
parabolic induction functor
\[
I_{P^\sharp}^{G^\sharp} : \Rep^{\fs_L} (L^\sharp) \to \Rep^{\fs} (G^\sharp) 
\]
corresponds to induction from the last algebra in part (a) to 
\[
\bigoplus\nolimits_{a \in [L / H_\lambda]} \Big( \cH (T_\fs^\sharp,W_\fs,q_\fs) 
\otimes \End_\C (V_\mu) \Big)^{X^L (\fs) / X^L (\omega,V_\mu)} 
\rtimes \mf R_\fs^\sharp .
\]
}  
\end{thm}
\begin{proof}
Part (a) is a consequence of \cite[(169)]{ABPS4} and \cite[Lemma 4.8]{ABPS4}, which
shows that 
\[
\End_\C (\C [L /H_\lambda]) \cong \C [L /H_\lambda] \otimes \C [X^L (\omega,V_\mu)] .
\]
The analogue of part (b) for $L$ and $G$ says that 
\[
I_P^G : \Rep^{\fs_L} (L) \to \Rep^{\fs} (G) 
\]
corresponds to induction from
\begin{align*}
& e^\fs_L \cH (L) e^\fs_L \cong \mc O (T_\fs) \otimes \End_\C (V_\mu \otimes 
\C [L / H_\lambda]) \quad \text{to} \\
& e^\sharp_{\lambda_G} \cH (G) e^\sharp_{\lambda_G} \cong
\cH (T_\fs,W_\fs,q_\fs) \otimes \End_\C (V_\mu \otimes 
\C [L / H_\lambda] \otimes \C \mf R_\fs^\sharp) . 
\end{align*}
To see that it is true, we reduce with \cite[Theorem 4.5]{ABPS4} to the algebras
\begin{align*}
& e_{\lambda_L} \cH (L) e_{\lambda_L} \cong \mc O (T_\fs) \otimes \End_\C (V_\lambda) , \\
& e_{\lambda_G} \cH (G) e_{\lambda_G} \cong
\cH (T_\fs,W_\fs,q_\fs) \otimes \End_\C (V_\lambda) . 
\end{align*}
Then we are in the situation where $(K_G,\lambda_G)$ is a cover of a type 
$(K_L,\lambda_L)$, and the statement about the induction functors follows from
\cite[(126)]{ABPS4} and \cite[Corollary 8.4]{BuKu3}. 

We note that here, for a given algebra homomorphism $\phi : A \to B$, we must 
use induction in the version $\Ind_A^B (M) = \Hom_A (B,M)$. However, in all the
cases we encounter $B$ is free of finite rank as a module over $A$ and it is
endowed with a canonical anti-involution 
\[
f \mapsto [f^\vee : g \mapsto \overline{f (g^{-1})}] .
\] 
Hence we may identify $\Hom_A (B,M) \cong B^* \otimes_A M \cong B \otimes_A M$.
This shows the desired claim for $I_P^G$.

Since $G^\sharp / P^\sharp \cong G / P ,\; I_P^G = I_{P^\sharp}^{G^\sharp}$ on 
$\Rep^{\fs_L} (L)$. In particular
\begin{equation}\label{eq:6.4}
\Res_{G^\sharp}^G \circ I_P^G = I_{P^\sharp}^{G^\sharp} \circ \Res_{L^\sharp}^L
\quad \text{as functors} \quad \Rep^{\fs_L} (L) \to \Rep^\fs (G^\sharp).
\end{equation}
The functor $\Res^L_{L^\sharp}$ corresponds to $\Res^{e^\sharp_L \cH (L) 
e^\sharp_L}_{e^\fs_{L^\sharp} \cH (L^\sharp) e^\fs_{L^\sharp}}$, and $\Res^G_{G^\sharp}$ 
to restriction from $e^\sharp_{\lambda_G} \cH (G) e^\sharp_{\lambda_G}$ to
$e^\sharp_{\lambda_{G^\sharp}} \cH (G^\sharp) e^\sharp_{\lambda_{G^\sharp}}$,
which is the algebra appearing in the statement of part (b).

The set $\mc H (W_\fs,q_\fs) \otimes \C [\mf R_\fs^\sharp]$ forms a basis for
$e^\sharp_{\lambda_G} \cH (G) e^\sharp_{\lambda_G}$ as a module over 
$e^\sharp_L \cH (L) e^\sharp_L$ and for $e^\sharp_{\lambda_{G^\sharp}} 
\cH (G^\sharp) e^\sharp_{\lambda_{G^\sharp}}$ as a module over 
$e^\fs_{L^\sharp} \cH (L^\sharp) e^\fs_{L^\sharp}$. It follows that
\begin{equation}\label{eq:6.5}
\Res^{e^\sharp_{\lambda_G} \cH (G) e^\sharp_{\lambda_G}}_{e^\sharp_{\lambda_{G^\sharp}} 
\cH (G^\sharp) e^\sharp_{\lambda_{G^\sharp}}} \circ \mr{ind}_{e^\sharp_L \cH (L) 
e^\sharp_L}^{e^\sharp_{\lambda_G} \cH (G) e^\sharp_{\lambda_G}} =
\mr{ind}_{e^\fs_{L^\sharp} \cH (L^\sharp) e^\fs_{L^\sharp}}^{e^\sharp_{\lambda_{G^\sharp}} 
\cH (G^\sharp) e^\sharp_{\lambda_{G^\sharp}}} \circ \Res^{e^\sharp_L \cH (L) 
e^\sharp_L}_{e^\fs_{L^\sharp} \cH (L^\sharp) e^\fs_{L^\sharp}} .
\end{equation}
Comparing \eqref{eq:6.4} and \eqref{eq:6.5}, we find that 
\begin{equation}\label{eq:6.6}
I_{P^\sharp}^{G^\sharp} \circ \Res_{L^\sharp}^L \quad \text{corresponds to} \quad
\mr{ind}_{e^\fs_{L^\sharp} \cH (L^\sharp) e^\fs_{L^\sharp}}^{e^\sharp_{\lambda_{G^\sharp}} 
\cH (G^\sharp) e^\sharp_{\lambda_{G^\sharp}}} \circ \Res^{e^\sharp_L \cH (L) 
e^\sharp_L}_{e^\fs_{L^\sharp} \cH (L^\sharp) e^\fs_{L^\sharp}}
\end{equation}
under the Morita equivalences from Theorems \ref{thm:1.2} and \ref{thm:1.4}.a.
Here both inductions can be constructed entirely in $\mc H (G^\sharp)$. The two
sides of \eqref{eq:6.6} are then related by applications of the idempotents 
$e^\fs_{L^\sharp}$ and $e^\sharp_{\lambda_{G^\sharp}}$ (which are supported on $G^\sharp$).
Hence the correspondence \eqref{eq:6.6} preserves $L^\sharp$-subrepresentations.
Since every irreducible $L^\sharp$-representation appears as a summand of an 
$L$-representation, this implies the analogue of \eqref{eq:6.6} on the whole of 
$\Rep^{\fs_L}(L^\sharp)$.
\end{proof}

Let $\mf t^\sharp = [L^\sharp, \sigma^\sharp]_{G^\sharp}$ be an inertial equivalence 
class for $G^\sharp$, with $\mf t^\sharp \prec \fs = [L,\omega]_G$. We abbreviate 
$\phi_{\omega,X^L (\fs)} = \{ \phi_{\omega,\gamma} : \gamma \in X^L (\fs) \}$,
where $\phi_{\omega ,\gamma}$ is as in Lemma \ref{lem:4.1}.
By \eqref{eq:4.4} there is a unique $X^L (\fs)$-orbit
\begin{equation}\label{eq:4.31}
\phi_{\omega,X^L (\fs)} \, \rho \; \subset \; \Irr (\C [X^L (\omega),\kappa_\omega]) 
\end{equation}
such that $T_{\mf t^\sharp} = (T_\fs^\sharp \times \phi_{\omega,X^L (\fs)} \rho )
/ X^L (\fs)$. Then $\phi_{\omega,X^L (\fs)} \rho$ determines a unique summand $\C a$
of $\C [L / H_\lambda] \cong \bigoplus_{[L / H_\lambda]} \C$, namely the irreducible 
representation of $X^L (\omega,V_\mu)$ obtained by restricting $\rho$. Let 
$V_{\sigma^\sharp}$ be the intersection of $V_\mu$ with the subspace of $V_\omega$
on which $\sigma^\sharp$ is defined, and let $\mf R_{\mf t^\sharp}$ be its stabilizer in 
$\mf R_\fs^\sharp$. Then $\mf R_{\mf t^\sharp}$ is also the stabilizer of $\mf t^\sharp$
in $\mf R_\fs^\sharp$ and
\begin{equation}\label{eq:4.39}
W_{\mf t^\sharp} = W_\fs \rtimes \mf R_{\mf t^\sharp} ,
\end{equation}
by \cite[Lemma 2.3]{ABPS4}.
Via the formula \eqref{eq:4.30} the operators $J(\gamma,\omega) 
\big|_{V_{\sigma^\sharp}}$ determine a 2-cocycle $\kappa'_\omega$ of the group
\begin{equation}\label{eq:4.34}
W' = \{ (w,\gamma) \in \Stab (\fs) : w \in W_{\mf t^\sharp} \}.
\end{equation}
Since \eqref{eq:4.30} is 1 on $W_\fs$, so is $\kappa'_\omega$.
By \eqref{eq:1.13} $W' / X^L (\fs) \cong W_{\mf t^\sharp}$. As $V_{\sigma^\sharp}$
is associated to the single $X^L (\fs)$-orbit \eqref{eq:4.31}, 
$\kappa'_\omega ((w,\gamma),(w',\gamma'))$ depends only on $(w,w')$. Thus it
determines a 2-cocycle $\kappa_{\sigma^\sharp}$ of $W_{\mf t^\sharp}$, which 
factors through $\mf R_{\mf t^\sharp} \cong W_{\mf t^\sharp} / W_\fs$.

\begin{lem}\label{lem:4.10}
\enuma{
\item The bijections in Theorem \ref{thm:4.9} restrict to 
\begin{align*}
& \Irr^{\mf t^\sharp} (G^\sharp) \longleftrightarrow 
(T_{\mf t^\sharp} \q W_{\mf t^\sharp} )_{\kappa_{\sigma^\sharp}} , \\
& \Irr_\temp^{\mf t^\sharp} (G^\sharp) \longleftrightarrow 
(T_{\mf t^\sharp,\uni} \q W_{\mf t^\sharp} )_{\kappa_{\sigma^\sharp}} ,
\end{align*}
where $T_{\mf t^\sharp,\uni}$ denotes the space of unitary representations
in $T_{\mf t^\sharp}$.
\item Suppose $\pi \in \Irr_\temp^{\mf t^\sharp} (G^\sharp)$ corresponds to
$[t,\rho]$ and has cuspidal support\\ 
$W_{\mf t^\sharp} (\chi \otimes \sigma^\sharp) \in T_{\mf t^\sharp} / 
W_{\mf t^\sharp}$. Then $W_{\mf t^\sharp} t$ is the unitary 
part of $\chi \otimes \sigma^\sharp$, with respect to the polar decomposition 
\[
T_{\mf t^\sharp} = T_{\mf t^\sharp,\uni} \times \Hom_\Z (X^* (T_{\mf t^\sharp}),
\R_{>0} ).
\]
}
\end{lem}
\begin{proof}
(a) Recall that $\Irr^{\mf t^\sharp}(G^\sharp)$ consists of those irreducible
representations that are contained in $I_{P^\sharp}^{G^\sharp}(\chi \otimes
\sigma^\sharp)$ for some $\chi \otimes \sigma^\sharp \in T_{\mf t^\sharp}$.
In Theorem \ref{thm:1.4}.b we translated $I_{P^\sharp}^{G^\sharp}$ to
induction between two algebras. The first one, Morita equivalent with
$\cH (L^\sharp)^{\fs_L}$, was
\[
\C [L / H_\lambda] \otimes \Big( \mc O (T_\fs^\sharp) \otimes \End_\C (V_\mu)
\Big)^{X^L (\fs)} .
\]
The second algebra, Morita equivalent with $\cH (G^\sharp)^\fs$, was
\[
\C [L / H_\lambda] \otimes \Big( \cH (T_\fs^\sharp, W_\fs,q_\fs) \otimes 
\End_\C (V_\mu) \Big)^{X^L (\fs)} \rtimes \mf R_\fs^\sharp .
\] 
As above, the $L^\sharp$-representation $\sigma^\sharp$ determines a summand
$\C a$ of $\C [L / H_\lambda]$ and a $X^L (\fs)$-stable subspace $V_{\sigma^\sharp}
\subset V_\mu$. Consequently $\mc H (L^\sharp )^{[L^\sharp, \sigma^\sharp]}$ is
Morita equivalent with the two algebras
\begin{equation}\label{eq:6.7}
\Big( \mc O (T_\fs^\sharp) \otimes \End_\C (V_{\sigma^\sharp}) \Big)^{X^L (\fs)} 
\quad \text{and} \quad \Big( \mc O (T_\fs^\sharp) \otimes 
\End_\C (\C \mf R_\fs^\sharp \cdot V_{\sigma^\sharp}) \Big)^{X^L (\fs)} .
\end{equation}
In these terms Theorem \ref{thm:1.4} shows that $\mc H (G^\sharp )^{\mf t^\sharp}$ is
is Morita equivalent with
\begin{equation}\label{eq:4.35}
\Big( \cH (T_\fs^\sharp, W_\fs,q_\fs) \otimes \End_\C (\C \mf R_\fs^\sharp 
\cdot V_{\sigma^\sharp}) \Big)^{X^L (\fs)} \rtimes \mf R_\fs^\sharp . \\
\end{equation}
Here the subspaces $w V_{\sigma^\sharp}$ with $w \in \mf R_\fs^\sharp$ are permuted
transitively by $\mf R_\fs^\sharp$, so upon taking $\mf R_\fs^\sharp$-invariants
only the $\mf R_{\mf t^\sharp}$ on $V_{\sigma^\sharp}$ survives. Recall that, for
any finite group $\Gamma$ and any $\Gamma$-algebra $A$:
\begin{equation}\label{eq:6.8}
A \rtimes \Gamma \cong \big( A \otimes \End_\C (\C \Gamma) \big)^\Gamma .
\end{equation}
Applying \eqref{eq:6.8} to \eqref{eq:4.35} first with $\Gamma = \mf R_\fs^\sharp$ and 
subsequently with $\Gamma = \mf R_{\mf t^\sharp}$ (in the opposite direction, taking the 
above transitivity into account), we find that \eqref{eq:4.35} is Morita equivalent with
\begin{equation}
\Big( \cH (T_\fs^\sharp, W_\fs,q_\fs) \otimes \End_\C (V_{\sigma^\sharp}) 
\Big)^{X^L (\fs)} \rtimes \mf R_{\mf t^\sharp} .
\end{equation}
The constructions in Section \ref{sec:geomEquiv} restrict to stratified
equivalences between \eqref{eq:4.35} and 
\begin{equation}\label{eq:4.32}
\begin{aligned}
& \big( \mc O (T_\fs^\sharp) \otimes \End_\C (\C \mf R_\fs^\sharp 
\cdot V_{\sigma^\sharp}) \big)^{X^L (\fs)} \rtimes W_\fs^\sharp , \\
& \big( \mc O (T_\fs^\sharp) \otimes \End_\C (V_{\sigma^\sharp}) 
\big)^{X^L (\fs)} \rtimes W_{\mf t^\sharp} . 
\end{aligned}
\end{equation}
By \eqref{eq:6.7}
\begin{equation}\label{eq:4.33}
\Irr \Big( \big( \mc O (T_\fs^\sharp) \otimes \End_\C (V_{\sigma^\sharp}) 
\big)^{X^L (\fs)} \Big) \cong T_{\mf t^\sharp} .
\end{equation}
As explained above with \eqref{eq:4.34}, the 2-cocycle $\kappa_\omega$ of 
$\Stab (\fs)$ reduces to the 2-cocycle $\kappa_{\sigma^\sharp}$ for the 
action of $W_{\mf t^\sharp}$ in \eqref{eq:4.32}. Now we apply Lemma 
\ref{lem:B.6}.a to \eqref{eq:4.32} and we find the first bijection. 
To obtain the second bijection, we use Theorem \ref{thm:4.9}.c.\\
(b) For the stratified equivalence between 
\[
\cH (T_\fs^\sharp, W_\fs,q_\fs) \otimes \End_\C (V_{\sigma^\sharp}) \text{ and } 
\mc O (T_\fs^\sharp) \otimes \End_\C (V_{\sigma^\sharp}) \rtimes W_\fs 
\]
the analogous claim about the cuspidal support is property (3) of Theorem
\ref{thm:4.6}. Clifford theory relates the irreducible representations of these 
algebras to those of \eqref{eq:4.35} and \eqref{eq:4.32}, in a way already 
discussed after \eqref{eq:4.29}. This implies that the desired property of the 
cuspidal support persists to the stratified
equivalence between \eqref{eq:4.35} and \eqref{eq:4.32}, which underlies part (a).
\end{proof}

\section{Relation with the local Langlands correspondence}
\label{sec:LLC}

We show how the local Langlands correspondence (LLC) for $G$ and $G^\sharp$ can be 
reconstructed in terms of twisted extended quotients. 

Let $\mb W_F$ be the Weil group of the local non-archimedean field $F$.
Recall that the Langlands dual group of $G = \GL_m (D)$ is $\check G = 
\GL_{md}(\C)$. A Langlands parameter for $G$ is continuous group homomorphism
$\phi : \mb W_F \times \SL_2 (\C) \to \check G$ such that:
\begin{itemize}
\item $\phi \big|_{\SL_2 (\C)}$ is a homomorphism of algebraic groups.
\item $\phi (\mb W_F)$ consists of semisimple elements.
\item $\phi$ is relevant for $G$: if $\check L$ is a Levi subgroup of $\check G$
which contains im$(\phi)$ and is minimal for that property, then (the conjugacy
class of) $\check L$ corresponds to (the conjugacy class of) a Levi subgroup of $G$.
\end{itemize}
We denote the collection of Langlands parameters for $G$, modulo conjugation by
$\check G$, by $\Phi (G)$. 

Every smooth character of $G$ is of the form $\nu \circ \Nrd$, with $\nu$ a 
smooth character of $F^\times$. Via Artin reciprocity it determines a Langlands
parameter (trivial on $\SL_2 (\C)$)
\begin{equation}\label{eq:6.3}
\hat \nu : \mb W_F \to \C^\times \cong Z (\GL_{md}(\C)) . 
\end{equation}
For any $\phi \in \Phi (G) ,\; \phi \hat \nu$ is a well-defined element of 
$\Phi (G)$ because the image of $\hat \nu$ is central in $\check G$.

\begin{thm}\label{thm:6.5}
The local Langlands correspondence for $G$ is a canonical bijection 
\[
\mathrm{rec}_{D,m} : \Irr (G) \to \Phi (G)
\]
with the following properties:
\enuma{
\item $\pi \in \Irr (G)$ is tempered if and only if $\mathrm{rec}_{D,m}(\pi)$ 
is bounded, that is, if \\
$\mathrm{rec}_{D,m}(\pi) (\mb W_F)$ is a bounded subset of $\check G$.
\item The L-packet $\Pi_\phi (G)$ is the single representation 
$\mathrm{rec}^{-1}_{D,m}(\phi)$.
\item $\mathrm{rec}_{D,m}$ is equivariant for the two actions of 
$\Irr (G / G^\sharp)$: on $\Irr (G)$ by twisting with smooth characters and on 
$\Phi (G)$ by multiplication with central Langlands parameters as in \eqref{eq:6.3}.
}
\end{thm}
\begin{proof}
For the bijection and part (a) see \cite[\S 11]{HiSa} and \cite[\S 2]{ABPS3}.
Ultimately it relies on the Jacquet--Langlands correspondence from \cite{DKV,Bad}.\\
(b) This is a direct consequence of the bijectivity.\\
(c) Since $\mathrm{rec}_{D,m}$ is determined completely by its behaviour on 
essentially square integrable representations of Levi subgroups of $G$ 
\cite[(13)]{ABPS3}, it suffices to prove (c) for such representations. Via the 
Jacquet--Langlands correspondence the issue can be transferred to $\Irr (\GL_n (F))$ 
with $n \leq md$. For general linear groups (c) is a well-known property of the LLC, 
and in fact constitutes a starting point of the construction, confer \cite[1.2]{Hen}.
\end{proof}

For $\fs = [L,\omega]_G$ we define $\Phi (G)^\fs$ as the image of $\Irr^\fs (G)$
under the bijection $\mathrm{rec}_{D,m}$. Similarly we define $\Phi (L)^{\fs_L} \subset \Phi (L)$.

\begin{lem}\label{lem:6.1}
The $\LLC$ for $G$ fits in a commutative diagram of canonical bijections
\[
\xymatrix{
\Irr^\fs (G) \ar@{->}[rr]^{\mathrm{rec}_{D,m}} \ar@{<->}[d] & & 
\Phi (G)^\fs \ar@{<->}[d] \\
T_\fs \q W_\fs  \ar@{<->}[rr] && \Phi (L)^{\fs_L} \q W_\fs 
}
\]
Here the bottom map comes from the $\LLC$ for $\Irr^{\fs_L} (L)$ and the left hand 
side comes from Theorem \ref{thm:4.6}. 
\enuma{
\item Suppose that $[\phi_L] \in \Phi (L)^{\fs_L}$
and that $\rho \in \Irr (W_{\fs,\phi_L})$ has as Springer parameter a unipotent 
class $[u] \in Z_{\check G_\fs}(\phi_L)$. Then there is a representative $u$ such 
that the right hand side sends $[\phi_L,\rho]$ to a Langlands parameter $\phi$ with
$\phi \big|_{\mb W_F} = \phi_L \big|_{\mb W_F}$ and $\phi (1,\matje{1}{1}{0}{1})
= \phi_L (1,\matje{1}{1}{0}{1}) u$.
\item Two elements $[t,\rho]$ and $[t',\rho']$ of $T_\fs \q W_\fs$ map to the same
$G$-representation if and only if there exists a $w \in W_\fs$ such that $w t = t'$
and the $W_{\fs,t}$-representations $\rho , w \rho'$ have Springer parameters 
involving the same unipotent class.
}
\end{lem}
\begin{proof}
The top horizontal and left vertical maps have already been established as 
bijective and canonical. The LLC for $L$ is the Cartesian product of the LLCs for the 
factors of $L$. Hence it is $W(G,L)$-equivariant, and its restriction to 
$T_\fs \longleftrightarrow \Phi (L)^{\fs_L}$ is $W_\fs$-equivariant. The canonicity
and bijectivity of the LLC for $L$ are inherited by the bottom horizontal map in the
diagram. This leaves a unique, canonical way to complete the
commutative diagram. \\
(a) To work out the map on the right hand side, it suffices to consider 
\[
L = \prod\nolimits_i L_i^{e_i} \quad \text{and} \quad 
\omega = \prod\nolimits_i \omega_i^{e_i}
\]
such that $(L_i,\omega_i)$ is not isomorphic to $(L_j,\omega_j)$ for $i \neq j$. 
Let $\phi_i : \mb W_F \times \SL_2 (\C) \to \GL_{m_i d}(\C)$ be a Langlands 
parameter for $\omega_i$. Then 
\[
\phi_L = \prod\nolimits_i \phi_i^{e_i} : \mb W_F \times \SL_2 (\C) \to 
\prod\nolimits_i \GL_{m_i d}(\C)^{e_i}
\]
is a Langlands parameter for $\omega$. We have $W_{\fs,\phi_L} = \prod_i S_{e_i}$,
where $S_{e_i}$ is embedded in $N_{\GL_{e_i d_i m}(\C)}(\GL_{m_i d}(\C)^{e_i})$
as permutation matrices. The unipotent class
\[
[u] = [\prod\nolimits_i u_i] \in 
\prod\nolimits_i \GL_{e_i m_i d}(\C) \subset Z_{\check G_\fs}(\phi_L)
\]
is determined by the standard Levi subgroup in which it is distinguished, say
\[
\check M = \prod\nolimits_{i,j} \GL_{b_{ij} m_i d}(\C)^{c_{ij}} 
\quad \text{ with } \quad \sum\nolimits_j c_{ij} b_{ij} = e_i .
\]
Assume for the moment that $\omega$ is tempered. By Theorem \ref{thm:4.6}
$[\omega,\rho] \in T_{\fs,\uni} \q W_\fs$ corresponds to $I_{PM}^G (\delta)$,
where 
\[
\delta = \prod\nolimits_{i,j} \delta_{ij}^{c_{ij}} \in \Irr^{[L,\omega]_M}_\temp (M)
\]
is the unique square-integrable modulo centre representation such that 
$W_{\fs,M} \omega$ is the unitary part of the cuspidal support of $\delta$. 
By construction \cite[\S 2]{ABPS4} the Langlands parameter $\phi$ of 
$I_{PM}^G (\delta)$ is the same as that of $\delta$, namely
$\phi = \prod_{i,j} \phi_{ij}^{c_{ij}}$ with $\phi_{ij} \big|_{\mb W_F} =
\phi_i^{b_{ij}} \big|_{\mb W_F}$ and 
\[
\phi_{ij} (1,\matje{1}{1}{0}{1}) = \phi_i (1,\matje{1}{1}{0}{1})^{b_{ij}} u_{ij} 
\]
where $u_{ij}$ is a distinguished unipotent element in 
$Z_{\GL_{b_{ij} m_i d}(\C)}(\GL_{m_i d}(\C)^{b_{ij}})$. Thus 
$\phi (1,\matje{1}{1}{0}{1})$ is distinguished in $\check M$ and $\phi$ has the
asserted shape. 

The general case, where $\omega$ is not necessarily tempered, follows from the
tempered case. The reason is that all the maps in the commutative diagram (a priori
except the right hand side) can be obtained from their tempered parts by some
kind of analytic continuation, as in \cite{ABPS1} and Theorem \ref{thm:4.6}.\\
(b) By Theorem \ref{thm:6.5}.b the elements of $T_\fs \q W_\fs$
are in bijection with the L-packets in $\Irr^\fs (G)$. Two elements $[t,\rho]$
and $[t',\rho']$ are equal if and only if there is a $w \in W_\fs$ such that
$w t' = t$ and $w \cdot \rho' = \rho$. We note also that for every $t \in T_\fs$ 
the group $W_{\fs,t} = W(R_{\fs,t})$ is product of symmetric groups. Hence all 
irreducible representations of $W_{\fs,t}$ are parametrized by different unipotent 
classes in a connected complex reductive group with maximal torus $T_\fs$ and 
root system $R_{\fs,t}$. So the condition becomes that $\rho$ and $w \cdot \rho'$
have the same unipotent class as Springer parameter.
\end{proof}

Let $\Irr_\cusp (L)$ be the space of supercuspidal $L$-representations and let
$\Phi (L)_\cusp$ be its image in $\Phi (L)$. The Weyl group 
\[
W(G,L) = N_G (L) / L \cong N_{\check G}(\check L) / \check L 
\]
acts naturally on both sets.

\begin{thm}\label{thm:6.2}
Let $\mathcal L$ be a set of representatives for the conjugacy classes of Levi
subgroups of $G$. The maps from Lemma \ref{lem:6.1} combine to a commutative
diagram of canonical bijections
\[
\xymatrix{
\Irr (G) \ar@{->}[rr]^{\mathrm{rec}_{D,m}} \ar@{<->}[d] & & 
\Phi (G) \ar@{<->}[d] \\
\bigsqcup\nolimits_{L \in \mathcal L} 
\Irr_\cusp (L) \q W (G,L)  \ar@{<->}[rr] && 
\bigsqcup\nolimits_{L \in \mathcal L} \Phi (L)_\cusp \q W (G,L) 
}
\]
Here the tempered representations correspond to the bounded Langlands parameters.
\end{thm}
\begin{proof}
The action of $W(G,L)$ on $L$ is simply by permuting some direct factors of
$L$, and the same for $\check L$. Hence the canonical bijection
$\Irr (L) \leftrightarrow \Phi (L)$ is $W(G,L)$-equivariant. The group $W_\fs$
is defined as the stabilizer in $W(G,L)$ of $T_\fs = \Irr^{\fs_L}(L)$, and
by the above equivariance it is also the stabilizer of $\Phi^{\fs_L} (L)$.
Consequently
\begin{align*}
& \Irr_\cusp (L) \q W (G,L) \cong \bigsqcup\nolimits_{\fs = [L,\omega]_G}
T_\fs \q W_\fs , \\
& \Phi (L)_\cusp \q W (G,L) \cong  \bigsqcup\nolimits_{\fs = [L,\omega]_G}
\Phi^{\fs_L} (L) \q W_\fs .
\end{align*}
Now we simply take the union of the commutative diagrams of Lemma \ref{lem:6.1}.
The characterization of temperedness and boundedness comes from Theorems 
\ref{thm:6.5}.a and \ref{thm:4.9}.c.
\end{proof}

To formulate the LLC for $G^\sharp$, we need enhanced Langlands parameters. In fact
these are already present in the LLC for $G$, but there the enhancement can be
neglected without any problems.

Recall that a Langlands parameter for $G^\sharp = \GL_m (D)_\der$ is a homomorphism
$\phi : \mb W_F \times \SL_2 (\C) \to \mathrm{PGL}_{md}(\C)$ subject to the same 
requirements as a Langlands parameter for $G$. The set of such parameters modulo 
conjugation by $\check{G^\sharp} = \mathrm{PGL}_{md}(\C)$ is denoted 
$\Phi (G^\sharp)$. We note that the simply connected cover $\SL_{md}(\C)$ of 
$\mathrm{PGL}_{md}(\C)$ also acts by conjugation on Langlands parameters for $G^\sharp$. 

An enhancement of $\phi$ is an irreducible representation $\rho$ of 
$\pi_0 (Z_{\SL_{md}(\C)}(\phi))$. In order that $(\phi,\rho)$ is relevant for 
$G^\sharp$, an extra condition is needed. For this we have to regard $D$ as part 
of the data of $G^\sharp$, in other words, we must consider not just the inner form 
$G^\sharp$ of $\SL_{md}(F)$, but even the inner twist determined by $(G^\sharp,D)$. 
The Hasse invariant of $D$ gives a character $\chi_D$ of 
$Z (\SL_{md}(\C)) \cong \Z / md \Z$ with kernel $m \Z / md \Z$. Notice that, by 
Schur's lemma, every enhancement $\rho$ of $\phi$ determines a character of 
$Z(\SL_{md}(\C))$. We define an enhanced Langlands parameter for $G^\sharp = 
\GL_m (D)_\der$ as a pair $(\phi,\rho)$ such that $\rho \big|_{Z(\SL_{md}(\C))} = 
\chi_D$. The collection of these,
modulo conjugation by $\SL_{md}(\C)$, is denoted $\Phi_e (G^\sharp)$. 

The LLC for $G^\sharp$ \cite{ABPS3} is a bijection
\begin{equation}\label{eq:6.1}
\Phi_e (G^\sharp) \longleftrightarrow \Irr (G^\sharp) : 
(\phi,\rho) \mapsto \pi (\phi,\rho) .
\end{equation}
such that
\begin{enumerate}[(i)]
\item if $\phi$ lifts to a Langlands parameter $\tilde \phi$ for $G$, 
then $\pi (\phi,\rho)$
is a direct summand of $\Res^G_{G^\sharp}(\mathrm{rec}_{D,m}^{-1} (\tilde \phi))$,
\item $\pi (\phi,\rho)$ is tempered if and only if $\phi$ is bounded,
\item the L-packet 
\[
\Pi_\phi (G^\sharp) = \{ \pi (\phi,\rho) : \rho \in 
\Irr \big( \pi_0 (Z_{\SL_{md}(\C)}(\phi)) \big) , 
\rho \big|_{Z(\SL_{md}(\C))} = \chi_D \} 
\]
is canonically determined.
\end{enumerate}

As $\Irr^\fs (G^\sharp)$ is defined in terms of restriction from $\Irr^\fs (G)$, 
it is a union of L-packets for $G^\sharp$. With (i) it canonically determines a set 
$\Phi_e (G^\sharp)^\fs$ of enhanced Langlands parameters for $G^\sharp$.

In the same way as for $G$, the LLC for a Levi subgroup $L^\sharp = L \cap G^\sharp$ 
follows from that for $L = \prod_i \GL_{m_i} (D)$. It involves enhancements 
from the action of 
\[
(\check{L^\sharp})_{sc} = \SL_{md}(\C) \cap \prod\nolimits_i \GL_{m_i d}(\C) .
\]
Given $\fs_L = [L,\omega]_L,\; \Irr^{\fs_L}(L^\sharp)$ is a union of L-packets for 
$L^\sharp$. Hence the corresponding set $\Phi_e (L^\sharp)^{\fs_L}$ of enhanced 
Langlands parameters is well-defined.

\begin{lem}\label{lem:6.3}
The $\LLC$ for $G^\sharp$ and the maps from Lemma \ref{lem:4.7} and 
Theorem \ref{thm:4.9}.b fit in the following commutative bijective diagram:
\[
\xymatrix{
\Irr^\fs (G^\sharp) \ar@{<->}[r] \ar@{<->}[d] &
\Phi_e (G^\sharp)^\fs \ar@{<->}[d] \\
(\Irr^{\fs_L}(L^\sharp) \q W_\fs^\sharp )_{\kappa_\omega} 
\ar@{<->}[r] \ar@{<->}[d] &
(\Phi_e (L^\sharp)^{\fs_L} \q W_\fs^\sharp )_{\kappa_\omega} \ar@{<->}[d] \\
\Big( \big( T_\fs \q X^L (\fs) X_\nr (L / L^\sharp) \big)_{\kappa_\omega} \q
W_\fs^\sharp \Big)_{\kappa_\omega} \ar@{<->}[r] \ar@{<->}[d] &
\Big( \big( \Phi (L)^{\fs_L} \q X^L (\fs) X_\nr (L / L^\sharp) 
\big)_{\kappa_\omega} \q W_\fs^\sharp \Big)_{\kappa_\omega} \ar@{<->}[d] \\
\Big( T_\fs \q \Stab (\fs) X_\nr (L / L^\sharp) \Big)_{\kappa_\omega} 
\ar@{<->}[r] \ar@{<->}[d] & \Big( \Phi (L)^{\fs_L} \q 
\Stab (\fs) X_\nr (L / L^\sharp) \Big)_{\kappa_\omega} \ar@{<->}[d] \\
\Big( \big( T_\fs \q W_\fs \big) \q \Stab (\fs)^+ X_\nr (L / L^\sharp) 
\Big)_{\kappa_\omega} \ar@{<->}[r]  &
\Big( \big( \Phi (L)^{\fs_L} \q W_\fs \big) \q
\Stab (\fs)^+ X_\nr (L / L^\sharp) \Big)_{\kappa_\omega}  
}\] 
All these maps are canonical up to permutations within L-packets.
In the last row the collection of $L$-packets is in bijection with
$\big( T_\fs \q W_\fs \big) / \Stab (\fs)^+ X_\nr (L / L^\sharp)$ and with
$\big( \Phi (L)^{\fs_L} \q W_\fs \big) / \Stab (\fs)^+ X_\nr (L / L^\sharp)$.
\end{lem}
\begin{proof}
The bijection between the first and the fourth set on the left hand side is given by 
Theorem \ref{thm:4.9}.b. Then Corollary \ref{cor:B.2} and \eqref{eq:4.27} give 
bijections to the third and fifth sets on the left, as the 2-cocycle $\kappa_\omega$ 
is by construction \eqref{eq:4.30} trivial on $W_\fs$. The bijection between the 
second and third sets on the left comes from Lemma \ref{lem:4.7}.a. By Lemma 
\ref{lem:4.7}.b it is canonical up to permutations within L-packets.

The LLC for $L$ is equivariant for permutations of the direct factors of $L$ and for 
twisting with characters of $L$ (because the LLC for $\GL_m (D)$ is so). This gives 
the three lower horizontal bijections. Applying Corollary \ref{cor:B.2} to the 
three lower terms on the right hand side gives bijections between them, and shows 
that the two lower squares in the diagram are canonical and commutative.

Similarly the LLC for $L^\sharp$ is equivariant for the action of $W_\fs^\sharp$, 
which leads to the second horizontal bijection. We define the upper two maps on 
the right hand side as the unique bijections that make the diagram commute. Since 
all the other maps in the upper two squares are canonical up to permutations within 
L-packets, so are the last two.

An L-packet for $G^\sharp$ consists of the irreducible $G^\sharp$-constituents
of an irreducible $G$-representation. In view of Lemma \ref{lem:6.1}, the
collection of L-packets in $\Irr^\fs (G^\sharp)$ is canonically in bijection 
with $T_\fs \q W_\fs$. From \eqref{eq:3.4} we can see how
\[
\Big( \big( T_\fs \q W_\fs \big) \q \Stab (\fs)^+ X_\nr (L / L^\sharp) 
\Big)_{\kappa_\omega}
\]
is constructed on the level of representations. We take an element $\pi \in 
\Irr^\fs (G)$ and transform it to an irreducible representation of $\cO (T_\fs) 
\rtimes W_\fs$ by a geometric equivalence. Then we form the twisted extended 
quotient by $\Stab (\fs)^+$, using Lemmas \ref{lem:B.6} and \ref{lem:B.7}, which 
corresponds to identifying $\pi$ with $\pi'$ if they have the same restriction to 
$G^\sharp Z(G)$, and decomposing $\pi$ in irreducible 
$G^\sharp Z(G)$-subrepresentations.
Finally we divide out the action of $X_\nr (L^\sharp Z(G) / L^\sharp)$, thus
identifying the $G^\sharp Z(G)$-representations with the same restriction to
$G^\sharp$. The implies the description of the L-packets in the lower left
term of the commutative diagram, and hence also in the lower right term.
\end{proof}

The bijection between the upper and the lower term on the right hand side of Lemma 
\ref{lem:6.3} can also be obtained as follows. First apply the recipe from Lemma 
\ref{lem:6.1} to $\Phi (L)^{\fs_L} \q W_\fs$, then take the twisted extended 
quotient with respect to $\Stab (\fs)^+$, and finally divide out the free action of 
$X_\nr (L^\sharp Z(G) / L^\sharp)$ to reach $\Phi_e (G^\sharp)^{\fs_L}$.

Let $R_{\mf t^\sharp,t}$ be the root system associated to $(G^\sharp ,t)$ be 
Harish-Chandra, by means of zeros of the $\mu$-function \cite[\S V.2]{Wal}.
Recall that the classical Springer correspondence was extended to Weyl groups of
disconnected complex reductive groups in \cite[Theorem 4.4]{ABPS5}.

\begin{lem}\label{lem:6.6}
Let $\mf t^\sharp = [L^\sharp,\sigma^\sharp]_{G^\sharp}$ be an inertial equivalence
class subordinate to $\fs = [L,\omega]_G$. Lemma \ref{lem:4.10}.a and the $\LLC$
for $G^\sharp$ and for $L^\sharp$ provide a commutative, bijective diagram
\[
\xymatrix{
\Irr^{\ft^\sharp} (G^\sharp) \ar@{<->}[r] \ar@{<->}[d] &
\Phi_e (G^\sharp)^{\ft^\sharp} \ar@{<->}[d] \\
(T_{\ft^\sharp} \q W_{\ft^\sharp} )_{\kappa_{\sigma^\sharp}} 
\ar@{<->}[r] &
(\Phi_e (L^\sharp)^{[L^\sharp,\sigma^\sharp]_{L^\sharp}} \q W_{\ft^\sharp} 
)_{\kappa_{\sigma^\sharp}}  
}\] 
Two elements $[t,\rho],[t',\rho'] \in (T_{\ft^\sharp} \q W_{\ft^\sharp} 
)_{\kappa_{\sigma^\sharp}}$ are mapped to $G^\sharp$-representations in the same
L-packet if and only if 
\begin{itemize}
\item $w t' = t$ for some $w \in W_{\ft^\sharp}$;
\item the $W_{\mf t^\sharp,t}$-representations $\rho$ and $w \cdot \rho'$ have 
parameters (in the Springer correspondence for possibly disconnected complex
reductive groups) with the same unipotent class, in the complex reductive group with
maximal torus $T_{\ft^\sharp}$, root system $R_{\mf t^\sharp,t}$ and Weyl group
$W_{\mf t^\sharp,t}$.
\end{itemize}
\end{lem}
\begin{proof}
The commutative diagram is obtained from Lemma \ref{lem:6.3}, taking \eqref{eq:4.4}
into account. 
To see whether $[t,\rho]$ and $[t',\rho']$ belong to the same L-packet, Lemma
\ref{lem:6.3} says that it suffices to look at their images in 
$\big( T_\fs \q W_\fs \big) / \Stab (\fs)^+ X_\nr (L / L^\sharp)$.

Let $\tilde t \in T_\fs$ be a lift of $t$. Then $W_{\ft^\sharp,t}$ is the isotropy
group of $X^L (\fs) X_\nr (L / L^\sharp) (\tilde\sigma^\sharp) \in T_{\ft^\sharp}$
in $W_\fs^\sharp$. Here $\sigma^\sharp$ is a projective representation of
\[
\big( X^L (\fs) X_\nr (L / L^\sharp) \big)_{\tilde t} = X^L (\omega).
\]
With Lemma \ref{lem:B.6} we get 
\[
\sigma^\sharp \rtimes \rho \in 
\Irr \big( \C [(\Stab (\fs) X_\nr (L/L^\sharp)_{\tilde t}, \kappa_\omega] \big) .
\]
The intersection of $(\Stab (\fs) X_\nr (L/L^\sharp)_{\tilde t}$ with
$W_\fs$ is $W_{\fs,\tilde t} = W (R_{\fs,\tilde t})$. Since $W_\fs$ commutes with
$X^L (\fs) X_\nr (L / L^\sharp)$, the restriction of $\sigma^\sharp \rtimes \rho$
to $W_{\fs,\tilde t}$ is $\dim (\sigma^\sharp)$ times 
$\rho \big|_{W_{\fs,\tilde t}}$. We want to show that 
\begin{equation}\label{eq:6.2}
R_{\mf t^\sharp,t} = R_{\fs,\tilde t} ,
\end{equation}
although in general $W_{\mf t^\sharp,t}$ is strictly larger than $W_{\fs,\tilde t}$.
Both root systems can be defined in terms of zeros of Harish-Chandra $\mu$-functions
associated to roots $\alpha \in R_\fs$. The function $\mu_\alpha$ (for $G$) is 
defined via intertwining operators betweeen $G$-representations, see 
\cite[\S IV.3 and \S V.2]{Wal}. These remain well-defined as intertwining operators
between $G^\sharp$-representations, which implies that $\mu_\alpha$ factors
through $T_\fs \to T_{\ft^\sharp}$ and in this way gives the function $\mu_\alpha$
for $G^\sharp$. By \cite[Theorem 1.6]{Sil2} all zeros of $\mu_\alpha$ are fixed
points of the reflection $s_\alpha \in W_\fs$. Hence $\mu_\alpha (t) \neq 0$ if
$s_\alpha (\tilde t) \neq \tilde t$, proving \eqref{eq:6.2}.

It follows that $[t,\rho]$ maps to $[\tilde t,\rho \big|_{W(R_{\ft^\sharp,t})}$ in
$\big( T_\fs \q W_\fs \big) / \Stab (\fs)^+ X_\nr (L / L^\sharp)$, and similarly
for $[t',\rho']$. The $\Stab (\fs)^+ X_\nr (L / L^\sharp)$-orbits of $[\tilde t,
\rho \big|_{W(R_{\ft^\sharp,t})}$ and $[\tilde t',\rho \big|_{W(R_{\ft^\sharp,t'})}$
are equal if and only if 
\[
\text{there is a } w \in W_{\ft^\sharp} \text{ such that } w t' = t \text{ and }
(w \rho')\big|_{W(R_{\ft^\sharp,t})} = \rho \big|_{W(R_{\ft^\sharp,t})}.
\]
By Lemma \ref{lem:6.1}.b the last condition is equivalent to $w \rho'$ and
$\rho$ having the same unipotent class as Springer parameter. Because $w$ is only
determined up to $W_{\ft^\sharp, t}$, these unipotent classes must be considered
in the complex reductive group with maximal torus $T_{\ft^\sharp}$, root system 
$R_{\mf t^\sharp,t}$ and Weyl group $W_{\mf t^\sharp,t}$.
\end{proof}

As before, let $\mathcal L$ be a set of representatives for the conjugacy classes 
of Levi subgroups of $G$. Then $\{ L^\sharp : L \in \mathcal L \}$ is a set of 
representatives for the conjugacy classes of Levi subgroups of $G^\sharp$.

\begin{thm}\label{thm:6.4}
The maps from Lemma \ref{lem:6.3} combine to a commutative diagram of bijections
\[
\xymatrix{
\Irr (G^\sharp) \ar@{<->}[r] \ar@{<->}[d] &
\Phi_e (G^\sharp) \ar@{<->}[d] \\
\bigsqcup\nolimits_{L \in \mathcal L} 
\Big( \Irr_\cusp (L^\sharp) \q W (G^\sharp,L^\sharp) \Big)_\natural \ar@{<->}[r] 
\ar@{<->}[d] & 
\bigsqcup\nolimits_{L \in \mathcal L} \Big( \Phi (L^\sharp)_\cusp \q 
W (G^\sharp,L^\sharp) \Big)_\natural \ar@{<->}[d] \\
\!\bigsqcup\nolimits_{L \in \mathcal L} 
\Big( \Irr_\cusp (L) \q \Irr (L / L^\sharp) W (G,L) \Big)_\natural \ar@{<->}[r] & 
\bigsqcup\nolimits_{L \in \mathcal L} \Big( \Phi (L)_\cusp \q \Irr (L/L^\sharp) 
W (G,L) \Big)_\natural 
} 
\]
Here the family of 2-cocycles $\natural$ restricts to $\kappa_\omega$ on 
$\Irr^{[L,\omega]_L}(L)$. The tempered representations correspond to the 
bounded enhanced Langlands parameters and the entire diagram is canonical up to 
permutations within L-packets. 

\end{thm}
\begin{proof}
The upper square follows quickly from Lemma \ref{lem:6.3}, in the same way as 
Theorem \ref{thm:6.2} followed from Lemma \ref{lem:6.1}. 

Recall from Lemma \ref{lem:4.7} that 
\[                                                   
\Irr^{\fs_L}(L^\sharp) \text{ is in bijection with }
(T_\fs \q X^L (\fs) X_\nr (L / L^\sharp))_{\kappa_\omega}. 
\]
Here $X^L (\fs)$ is the stabilizer of $\fs_L = [L,\omega]_L$ in 
$\Irr (L / L^\sharp Z(G))$. A character of $L / L^\sharp$ which is ramified on $Z(G)$ 
cannot stabilize $\fs_L$, so $X^L (\fs) X_\nr (L / L^\sharp)$ is the stabilizer of
$\fs_L$ in $\Irr (L / L^\sharp)$. By Theorem \ref{thm:6.5} the LLC for $L$ is 
bijective and $\Irr (L / L^\sharp)$-equivariant, so $X^L (\fs) X_\nr (L / L^\sharp)$
is also the stabilizer of $\Phi (L)^{\fs_L}$ in $\Irr (L /L^\sharp)$. This implies
\begin{multline*}
( \Irr_\cusp (L) \q \Irr (L / L^\sharp) )\natural \cong \bigsqcup\nolimits_{\fs_L = 
[L,\omega]_L} (\Irr^{\fs_L}(L) \q X^L (\fs) X_\nr (L / L^\sharp) )_{\kappa_\omega} \\
\cong \bigsqcup\nolimits_{\fs_L = [L,\omega]_L} 
\Irr^{\fs_L}(L^\sharp) = \Irr_\cusp (L^\sharp) ,
\end{multline*}
and similarly for Langlands parameters. These bijections are equivariant for 
permutations of the direct factors of $L$, so applying $( - \q W (G,L))_{\kappa_\omega}$ 
to all of them produces a commutative square as in the theorem, but with lower row
\begin{multline*}
\bigsqcup\nolimits_{L \in \mathcal L} 
\Big( \big( \Irr_\cusp (L) \q \Irr (L / L^\sharp) \big)_\natural \q W (G,L) \Big)_\natural 
\longleftrightarrow \\
\bigsqcup\nolimits_{L \in \mathcal L} \Big( \big( \Phi (L)_\cusp \q \Irr (L/L^\sharp)  
\big)_\natural \q W (G,L) \Big)_\natural .
\end{multline*}
We apply Corollary \ref{cor:B.2} to get the row in the theorem. The canonicity of the 
thus obtained commutative diagram is a consequence of the analogous property in Lemma 
\ref{lem:6.3}. The temperedness/boundedness correspondence follows from the
properties of the local Langlands correspondences for $G,G^\sharp,L$ and $L^\sharp$.
\end{proof}

\begin{ex} \label{the_example}
{\rm
Let $G = \SL_5(D)$.   
Let $V_4$ denote the non-cyclic group of order $4$.   Let $\bW_F$ denote the Weil group of $F$.   
There exists a classical Langlands parameter $\phi$ which factors through $V_4$:
\begin{align}\label{WEil}
\phi \colon \bW_F \to V_4 \to \PGL_2(\C)
\end{align}
Let $\tau$ be the cuspidal representation of $D^\times$ which has, as its Langlands parameter, 
a lift of $\phi$ to $\GL_2(\C)$.   Consider the group of characters $\chi$ for which $\chi \tau \cong \tau$.    
This group is isomorphic to $V_4$ and comprises the four characters
$\{1,\gamma, \eta, \gamma \eta \}$, where $\gamma, \eta$ are quadratic characters.    Let
\begin{align*}
L & = (D^\times)^5 \cap \SL_5(D)\\
\sigma & = \tau \otimes 1 \otimes \gamma \otimes \eta \otimes  \gamma \eta \in \Irr(L)\\
\fs & = [L,\sigma]_G
\end{align*}
Twisting by $\eta$ corresponds to the permutation $(13)(24)$, twisting by $\gamma \eta$ 
corresponds to the permutation $(14)(23)$.   
The Bernstein finite group $W^\fs$ is isomorphic to $V_4$.     
The corresponding Bernstein variety is the quotient $T^\fs/V_4$, where $T^\fs$ has 
the structure of a complex torus of dimension $4$.   

The summand $\cH^\fs(G)$ of the Hecke algebra $\cH(G)$ is Morita equivalent to the twisted crossed product
\[
\mathcal{O}(T^\fs) \rtimes_{\natural}V_4 ,
\] 
where $\natural$ is the  $2$-cocycle associated to the above projective representation of  $V_4$.  
Following \cite{ABPS3},
$\Irr(\mathcal{O}(T^\fs) \rtimes_{\natural}V_4)$ is the twisted extended quotient 
$(T^\fs \q V_4)_{\natural}$.
Now consider the standard projection
\[
\pi^\fs: (T^\fs \q V_4)_{\natural} \to T^\fs / V_4 .
\]
Let $(V_4)_t$ denote the isotropy group of $t \in T^\fs$.   Let
\begin{align*}
X & = \{ t \in T^\fs : |(V_4)_t| = 2\} ,\\
Y & = \{t \in T^\fs : (V_4)_t = V_4\} .
\end{align*}
We note that
\begin{itemize}
\item on the complement of $(X \cup Y)/V_4$ the fibre of $\pi^\fs$ has cardinality $1$: the corresponding 
(parabolically) induced representation is irreducible
\item  on $X/V_4$ the fibre of $\pi^\fs$ has cardinality $2$: the 
corresponding induced representation has two inequivalent irreducible constituents 
\item on $Y/V_4$ the fibre of $\pi^\fs$ has cardinality $1$, because $V_4$ admits a unique 
irrreducible projective representation with cocycle $\natural$: 
the corresponding induced representation has two \emph{equivalent} constituents.
\end{itemize}
The geometric structure of $\Irr^\fs (\SL_5(D))$ is the variety $T^\fs / V_4$ with doubling on $X/V_4$.}
\end{ex}

\appendix
\section{Twisted extended quotients}
\label{app:extquot}

Let $\Gamma$ be a group acting on a topological space $X$. In 
\cite[\S 2]{ABPS6} we studied various extended quotients of $X$ by $\Gamma$. In this 
paper we need the most general version, the twisted extended quotients.

Let $\natural$ be a given function which assigns to each 
$x \in X$ a 2-cocycle 
\[
\natural_x \colon \Gamma_x \times \Gamma_x \to \C^\times \text{, where } 
\Gamma_x = \{\gamma \in \Gamma : \gamma x = x\}. 
\]
It is assumed that $\natural_{\gamma x}$ and $\gamma_*{\natural_x}$ define the same
class in $H^2 (\Gamma_x , \C^\times)$, where $\gamma_* : \Gamma_x \to \Gamma_{\gamma x}$
sends $\alpha$ to $\gamma \alpha \gamma^{-1}$. Define 
\[
\widetilde X_\natural : = \{(x,\rho) : x \in X, \rho \in \Irr \,\C[\Gamma_x, \natural_x] \}.
\]
We  require, for every $(\gamma,x) \in \Gamma \times X$, a definite algebra isomorphism
\[
\phi_{\gamma,x} : \C[\Gamma_x,\natural_x]  \to \C[\Gamma_{\gamma x},\natural_{\gamma x}]
\]
such that:
\begin{itemize}
\item $\phi_{\gamma,x}$ is inner if $\gamma x = x$;
\item $\phi_{\gamma',\gamma x} \circ \phi_{\gamma,x} = 
\phi_{\gamma' \gamma,x}$ for all $\gamma',\gamma \in \Gamma, x \in X$.
\end{itemize}
We call these maps connecting homomorphisms, because they are reminiscent of a connection
on a vector bundle. Then we can define a $\Gamma$-action on $\widetilde X_\natural$ by
\[
\gamma \cdot (x,\rho) = (\gamma x, \rho \circ \phi_{\gamma,x}^{-1}).
\]
We form the \emph{twisted extended quotient}
\[
(X\q \Gamma)_\natural : = \widetilde{X}_\natural/\Gamma.
\]
We note that this reduces to the extended quotient of the second kind $X \q \Gamma$ 
from \cite[\S 2]{ABPS6} if $\natural_x$ is trivial for all $x \in X$ and 
$\phi_{\gamma,x}$ is conjugation by $\gamma$.

The map 
\[
\widetilde X_\natural \to X,  \quad (x,\rho) \mapsto x
\]
induces a map 
 \[
\pi_1: (X\q \Gamma)_\natural \to X/\Gamma
 \]
which we will call the \emph{standard projection}.

Such twisted extended quotients typically arise in the following situation.
Let $A$ be a $\C$-algebra such that all irreducible $A$-modules have countable
dimension over $\C$. Let $\Gamma$ be a group acting on $A$ by
automorphisms and form the crossed product $A \rtimes \Gamma$.

Let $X = \Irr (A)$. Now $\Gamma$ acts on $\Irr (A)$ and we get $\natural$ 
as follows. Given $x \in \Irr (A)$ choose an irreducible representation  
$(\pi_x,V_x)$ whose isomorphism class is $x$. 
For each $\gamma \in \Gamma$ consider $\pi_x$ twisted by $\gamma$:
\[
\gamma \cdot \pi_x : a \mapsto \pi_x (\gamma^{-1} a \gamma).
\]
Then $\gamma \cdot x$ is defined as the isomorphism class of $\gamma \cdot \pi_x$.
Since $\gamma \cdot \pi_x$ is equivalent to $\pi_{\gamma x}$, there exists 
a nonzero intertwining operator 
\begin{equation}\label{eq:B.1}
T_{\gamma,x} \in \Hom_A (\gamma \cdot \pi_x , \pi_{\gamma x}) .
\end{equation}
By Schur's lemma (which is applicable because $\dim V_x$ is countable) $T_{\gamma,x}$ 
is unique up to scalars, but in general there is no preferred choice. 
For $\gamma, \gamma' \in \Gamma_x$ there exists a unique $c \in \C^\times$ such that
\[
T_{\gamma,x} \circ T_{\gamma',x} = c T_{\gamma \gamma',x}.
\]
We define the 2-cocycle by 
\[
\natural_x (\gamma,\gamma') = c.
\]
Let $N_{\gamma,x}$ with $\gamma \in \Gamma_x$ be the standard basis of 
$\C [\Gamma_x,\natural_x]$. The algebra homomorphism $\phi_{\gamma,x}$ is essentially 
conjugation by $T_{\gamma,x}$, but we must be careful if some of the $T_\gamma$
coincide. The precise definition is 
\begin{equation}\label{eq:twisting}
\phi_{\gamma,x} (N_{\gamma',x}) = \lambda N_{\gamma \gamma' \gamma^{-1},\gamma x} 
\quad \text{if} \quad T_{\gamma,x} T_{\gamma',x} T_{\gamma,x}^{-1} = 
\lambda T_{\gamma \gamma' \gamma^{-1}, \gamma x}, \lambda \in \C^\times .
\end{equation}
Notice that \eqref{eq:twisting} does not depend on the choice of $T_{\gamma,x}$.

Suppose that $\Gamma_x$ is finite and $(\tau,V_\tau) \in \Irr (\C [\Gamma_x,\natural_x])$.
Then $V_x \otimes V^*_\tau$ is an irreducible $A \rtimes \Gamma_x$-module,
in a way which depends on the choice of intertwining operators $T_{\gamma,x}$. 

\begin{lem} \label{lem:B.6} \textup{\cite[Lemma 2.3]{ABPS6}} \ \\
Let $A$ and $\Gamma$ be as above and assume that the action of $\Gamma$ on
$\Irr (A)$ has finite isotropy groups.
\enuma{
\item There is a bijection
\[
\begin{array}{ccc}
(\Irr (A) \q \Gamma)_\natural & \longleftrightarrow & \Irr (A \rtimes \Gamma) \\
(\pi_x,\tau) & \mapsto & \pi_x \rtimes \tau := 
\Ind_{A \rtimes \Gamma_x}^{A \rtimes \Gamma} (V_x \otimes V^*_\tau) .
\end{array}
\]
\item If all irreducible $A$-modules are one-dimensional, then part (a) 
becomes a natural bijection
\[
\Irr (A) \q \Gamma \longleftrightarrow \Irr (A \rtimes \Gamma) .
\]
}
\end{lem}

Via the following result twisted extended quotients also arise from algebras of
invariants.

\begin{lem}\label{lem:B.7}
Let $\Gamma$ be a finite group acting on a $\C$-algebra $A$. There is a bijection 
\[
\begin{array}{ccc}
\{ V \in \Irr (A \rtimes \Gamma) : V^\Gamma \neq 0 \} & \longleftrightarrow & 
\Irr (A^\Gamma) \\
V & \mapsto & V^\Gamma .
\end{array}
\]
If all elements of $\Irr (A)$ have countable dimension, it becomes
\[
\begin{array}{ccc}
\{ (\pi_x,\tau) \in (\Irr (A) \q \Gamma)_\natural : \Hom_{\Gamma_x}(V_\tau,V_x) \neq 0 \} 
& \longleftrightarrow & \Irr (A^\Gamma) \\
(\pi_x,\tau) & \mapsto & \Hom_{\Gamma_x}(V_\tau,V_x) .
\end{array}
\]
\end{lem}
\begin{proof}
Consider the idempotent 
\begin{equation}\label{eq:4.23}
p_\Gamma = |\Gamma|^{-1} \sum\nolimits_{\gamma \in \Gamma} \gamma \in \C [\Gamma] .
\end{equation}
It is well-known and easily shown that 
\[
A^\Gamma \cong p_\Gamma (A \rtimes \Gamma) p_\Gamma
\]
and that the right hand side is Morita equivalent with the two-sided ideal 
\[
I = (A \rtimes \Gamma) p_\Gamma (A \rtimes \Gamma) \subset A \rtimes \Gamma .
\]
The Morita equivalence sends a module $V$ over the latter algebra to
\[
p_\Gamma (A \rtimes \Gamma) \otimes_{(A \rtimes \Gamma) p_\Gamma
(A \rtimes \Gamma)} V = V^\Gamma . 
\]
As $I$ is a two-sided ideal, 
\[
\Irr (I) = \{ V \in \Irr (A \rtimes \Gamma) : I \cdot V \neq 0 \}  =
\{ V \in \Irr (A \rtimes \Gamma) : p_\Gamma V = V^\Gamma \neq 0 \}
\]
This gives the first bijection. From Lemma \ref{lem:B.6}.a we know that
every such $V$ is of the form $\pi_x \rtimes \tau$. With Frobenius reciprocity
we calculate
\[
(\pi_x \rtimes \tau)^\Gamma = \Big( \Ind_{A \rtimes \Gamma_x}^{A \rtimes \Gamma} 
(V_x \otimes V^*_\tau) \Big)^\Gamma \cong (V_x \otimes V^*_\tau)^{\Gamma_x} =
\Hom_{\Gamma_x}(V_\tau,V_x).
\]
Now Lemma \ref{lem:B.6}.a and the first bijection give the second.
\end{proof}

Let $A$ be a commutative $\C$-algebra all whose irreducible representations 
are of countable dimension over $\C$. Then $\Irr (A)$ consists of
characters of $A$ and is a $T_1$-space. Typical examples are $A = C_0 (X)$
(with $X$ locally compact Hausdorff), $A = C^\infty (X)$ (with $X$ a smooth
manifold) and $A = \mc O (X)$ (with $X$ an algebraic variety).

As a kind of converse to Lemmas \ref{lem:B.6} and \ref{lem:B.7}, we show that
every twisted extended quotient of $\Irr (A)$ appears as the space of 
irreducible representations of some algebras. With small modifications, the
argument also works for smooth manifolds and algebraic varieties.

Let $\Gamma$ be a group acting on $A$ by algebra automorphisms, such that 
$\Gamma_x$ is finite for every $x \in \Irr (A)$. Recall that every 2-cocycle
$\natural$ of
$\Gamma$ arises from a projective $\Gamma$-representation $(\mu,V_\mu)$ by
\[
\mu (\gamma) \mu (\gamma') = \natural(\gamma, \gamma') \mu (\gamma,\gamma') . 
\]
Let $\Gamma$ act on $A \otimes \End_\C (V_\mu)$ by
\[
\gamma \cdot (a \otimes h) = 
\gamma (a) \otimes \mu (\gamma) h \mu (\gamma )^{-1}.
\]
\begin{lem}\label{lem:B.1}
There are bijections 
\[
\begin{array}{lcl}
\Irr \big( (A \otimes \End_\C (V_\mu))\rtimes \Gamma \big) & 
\longleftrightarrow & ( \Irr (A) \q \Gamma )_\natural , \\
\Irr \big( (A \otimes \End_\C (V_\mu) )^\Gamma \big) & \longleftrightarrow & 
\{ [x,\rho] \in ( X \q \Gamma )_\natural : \rho \; \mathrm{ appears \; in } \; V_\mu \} .
\end{array}
\]
\end{lem}
\begin{proof}
We can identify $\Irr (A \otimes \End_\C (V_\mu))$ with $\{ \C_x \otimes V_\mu :
x \in \Irr (A) \}$. It follows directly from \eqref{eq:B.1} that we can take
$T_{\gamma,x} = \mu (\gamma)$ for all $\gamma \in \Gamma$ and $x \in \Irr (A)$.
Thus the first bijection is an instance of Lemma \ref{lem:B.6}.a.

Let $x \in \Irr (A)$ and $(\tau,V_\tau) \in \Irr (\C [\Gamma_x,\natural])$. Then
\[
\Hom_{\Gamma_x}(\tau,\C_x \otimes V_\mu) = \Hom_{\Gamma_x}(\tau,V_\mu) , 
\]
and this is nonzero if and only if $\tau$ appears in $V_\mu$. Now an application
of Lemma \ref{lem:B.7} proves the second bijection.
\end{proof}

\begin{cor}\label{cor:B.2}
In the above setting, suppose that $\Gamma = \Gamma_1 \rtimes \Gamma_2$ 
is a semidirect product. Then there is a canonical bijection
\[
(\Irr (A) \q \Gamma )_\natural \longleftrightarrow   
\big( (\Irr (A) \q \Gamma_1 )_\natural \q \Gamma_2 )_\natural .
\]
\end{cor}
\begin{proof}
The bijection is obtained from Lemma \ref{lem:B.1} and
\[
(A \otimes \End_\C (V_\mu) ) \rtimes \Gamma =  
\big( (A \otimes \End_\C (V_\mu) ) \rtimes \Gamma_1 \big) \rtimes \Gamma_2
\]
It is canonical because the same 2-cocycle is used on both sides.
\end{proof}

\end{document}